\documentclass[12pt]{amsart}
\usepackage{amscd,amsmath,amsthm,amssymb}
\usepackage{algorithm}
\usepackage{algpascal}
\usepackage{tikz}
\usepackage{graphicx}
\usepackage{multicol}
\usepackage{epic,eepic}
\usepackage{amsfonts,amssymb,amscd,amsmath,enumerate,verbatim}
\def\NZQ{\mathbb}               % the font for N,Z,Q,R,C
\def\NN{{\NZQ N}}

\def\KK{{\NZQ K}}
\def\frk{\frak}               % font for "Fraktur"

\def\Phi{{\frk n}}
\def\Phi{{\frk N}}
\def\opn#1#2{\def#1{\operatorname{#2}}} % to make operators
\opn\chara{char} \opn\length{\ell} \opn\pd{pd} \opn\rk{rk}
\opn\projdim{proj\,dim} \opn\injdim{inj\,dim} \opn\rank{rank}
\opn\depth{depth} \opn\grade{grade} \opn\height{height}
\opn\embdim{emb\,dim} \opn\codim{codim}

\opn\Tr{Tr} \opn\bigrank{big\,rank}
\opn\superheight{superheight}\opn\lcm{lcm}
\opn\trdeg{tr\,deg}%\emph{
\opn\reg{reg} \opn\lreg{lreg} \opn\ini{in} \opn\lpd{lpd}
\opn\size{size}\opn\bigsize{bigsize}
\opn\cosize{cosize}\opn\bigcosize{bigcosize}
\opn\sdepth{sdepth}\opn\sreg{sreg}
\opn\link{link}\opn\fdepth{fdepth}\opn\e{e}
\opn\div{div} \opn\Div{Div} \opn\cl{cl} \opn\Cl{Cl}
\opn\Spec{Spec} \opn\Supp{Supp} \opn\supp{supp} \opn\Sing{Sing}
\opn\Ass{Ass} \opn\Min{Min}\opn\Mon{Mon} \opn\dstab{dstab} \opn\astab{astab}
\opn\Syz{Syz}
\opn\Ann{Ann} \opn\Rad{Rad} \opn\Soc{Soc}
\opn\Im{Im} \opn\Ker{Ker} \opn\Coker{Coker} \opn\Am{Am} \opn\i{int}
\opn\Hom{Hom} \opn\Tor{Tor} \opn\Ext{Ext} %\opn\End{End}
\opn\Aut{Aut} \opn\id{id}

\opn\nat{nat}
\opn\pff{pf}%   \pf exists already
\opn\Pf{Pf} \opn\GL{GL} \opn\SL{SL} \opn\mod{mod} \opn\ord{ord} \opn\reg{reg}
\opn\Gin{Gin} \opn\Hilb{Hilb}\opn\sort{sort}\opn\width{width} \opn\lk{lk} \opn\del{del}
\opn\aff{aff} \opn\con{conv} \opn\relint{relint} \opn\st{st}
\opn\lk{lk} \opn\cn{cn} \opn\core{core} \opn\vol{vol}
\opn\link{link} \opn\star{star}\opn\lex{lex}
\opn\gr{gr}

\def\P{{\mathcal{P}}}
\def\pot#1#2{#1[\kern-0.28ex[#2]\kern-0.28ex]}

\opn\dirlim{\underrightarrow{\lim}}
\opn\inivlim{\underleftarrow{\lim}}
\def\Implies{\ifmmode\Longrightarrow \else
        \unskip${}\Longrightarrow{}$\ignorespaces\fi}
\def\implies{\ifmmode\Rightarrow \else
        \unskip${}\Rightarrow{}$\ignorespaces\fi}
\def\iff{\ifmmode\Longleftrightarrow \else
        \unskip${}\Longleftrightarrow{}$\ignorespaces\fi}

\let\:=\colon
\newtheorem{Theorem}{Theorem}[section]
 \newtheorem{Lemma}[Theorem]{Lemma}
 \newtheorem{Corollary}[Theorem]{Corollary}
 \newtheorem{Proposition}[Theorem]{Proposition}
 \newtheorem{Remark}[Theorem]{Remark}
 \newtheorem{Remarks}[Theorem]{Remarks}
 \newtheorem{Example}[Theorem]{Example}
 \newtheorem{Examples}[Theorem]{Examples}
 \newtheorem{Definition}[Theorem]{Definition}

\let\epsilon\varepsilon
\let\kappa=\varkappa
\def\qed{\ifhmode\textqed\fi
      \ifmmode\ifinner\quad\qedsymbol\else\dispqed\fi\fi}
\def\textqed{\unskip\nobreak\penalty50
       \hskip2em\hbox{}\nobreak\hfil\qedsymbol
       \parfillskip=0pt \finalhyphendemerits=0}
\def\dispqed{\rlap{\qquad\qedsymbol}}

\opn\dis{dis}
\def\pnt{{\raise0.5mm\hbox{\large\bf.}}}

\opn\Lex{Lex}

\begin{document}
 \title {Properties of the coordinate ring of a convex polyomino}

 \author {Claudia Andrei}
 \address{Faculty of Mathematics and Computer Science, University of Bucharest,  Str. Academiei 14, Bucharest – 010014, Romania, \emph{E-mail address: claudiaandrei1992@gmail.com}}

\begin{abstract}
We classify all convex polyomino whose coordinate rings are Gorenstein. We also compute the Castelnuovo-Mumford regularity of the coordinate ring of any stack polyomino in terms of the smallest interval which contains its vertices. We give a recursive formula for computing the multiplicity of a stack polyomino.
\end{abstract}
\thanks{}
\maketitle

\section*{Introduction}

A polyomino $\P$ is a finite connected set of adjacent cells in the cartesian plane $\mathbb{N}^2$. A cell in $\mathbb{N}^2$ means a unitary square. A polyomino $\P$ is said to be column convex (respectively row convex) if every column (respectively row) is connected. According to \cite{CR}, $\P$ is a convex polyomino if for every two cells of $\P$ there is a monotone path, that is a path having only two directions, between them contained in $\P$.
Convex polyominoes include one-sided ladders, $2$-sided ladders and stack polyominoes.

Let $\KK$ be a field and $S$ be the polynomial ring over $\KK$ in the variables $x_{ij}$ with $(i,j)$ vertex in $\P$. The polyomino ideal $I_{\P}$ is the ideal of $S$ generated by all $2$-inner minors of $\P$ and the coordinate ring of $\P$ is defined as the quotient ring $\KK[\P]=S/I_{\P}$. The ideal $I_{\P}$ and the ring $\KK[\P]$ were first studied by Qureshi in \cite{Q}. There it was shown that if $\P$ is a convex polyomino, then $\KK[\P]$ is a normal Cohen-Macaulay domain. This was proved by viewing the ring $\KK[\P]$ as the edge ring of a suitable bipartite graph $G_{\P}$ associated with $\P$.

Understanding the graded free resolution of polyomino ideals is a difficult task. A first step in this direction was done in \cite{EHH}, where the convex polyomino ideals which are linearly related or have a linear resolution are classified.

Our paper is organised as follows. In Section 1, we recall the basic terminology related to convex polyominoes and their associated bipartite graphs. The first main result of this paper appears in Section 2, where we classify all convex polyominoes whose coordinate rings are Gorenstein (Theorem~\ref{mainresult}). For this classification, we use a result due to Ohsugi and Hibi (\cite{HO}) who classified all $2$-connected bipartite graphs whose edge rings are Gorenstein. In the case of stack polyominoes, we recover the classification of all Gorenstein stack polyominoes given in \cite{Q}; see Section 3.

In Section 4, we compute the Castelnuovo-Mumford regularity of the coordinate ring of any stack polyomino in terms of the smallest interval which contains its vertices (Corollary~\ref{regularity}). As a byproduct, we get the $a$-invariant of $\KK[\P]$ by Theorem~\ref{p}. The computation of the regularity uses as an important tool the formula of the $a$-invariant of $\KK[G_{\P}]$ given in \cite{VV}.

Finally, in Section 5 we give a recursive formula for computing the multiplicity of $\KK[\P]$ if $\P$ is a stack polyomino and we show some concrete cases when this formula may be applied.

\section{Preliminaries}

\label{firstsection}
To begin with, we present some concepts and notations about collections of cells and polyominoes.

We consider on $\mathbb{N}^2$ the natural partial order defined as follows: $(i,j)\leq (k,l)$ if and only if $i\leq k$ and $j\leq l$. If $a,b\in \mathbb{N}^2$ with $a\leq b$, then the set \[[a,b]=\{c\in \NN^2\mid a\leq c\leq b\}\] is an interval in $\NN^2$. If $a=(i,j)$ and $b=(k,l)\in\NN^2$ have the property that $j=l$ (respectively $i=k$), then the interval $[a,b]$ is a called horizontal (respectively vertical) edge interval.

The interval \[C=[a, a+(1,1)]\] is called a cell in $\NN^2$ with lower left corner $a$. The elements of $C$ are called vertices of $C$ and we denote their set by $V(C)$. The set of edges of $C$ is \[E(C)=\{\{a,a+(0,1)\},\{a,a+(1,0)\},\{a+(0,1),a+(1,1)\},\{a+(1,0),a+(1,1)\}\}.\]

We consider $A$ and $B$ two cells in $\NN^2$ with lower left corners $(i,j)$ and $(k,l)$. Then the set \[[A,B]=\{E| E\text{ is a cell in }\NN^2 \text{ with lower left corner }(r,s) \text{ such that } \]\[i\leq r\leq k\text{ and } j\leq s\leq l\}\] is called a cell interval. In the case that $j=l$ (respectively $i=k$), the cell interval $[A,B]$ is called a horizontal (respectively vertical) cell interval.

Let $\P$ be a finite collection of cells of $\NN^2$. The vertex set of $\P$ and the edge set of $\P$ are
\[V(\P)=\cup_{C\in \P}V(C) \text{ and } E(\P)=\cup_{C\in\P}E(C).\]
Two cells $A$ and $B$ of $\P$ are connected, if there is a sequence of cells of $\P$ given by $A=A_1,A_2, \ldots, A_{n-1}, A_n=B$ such that $A_i\cap A_{i+1}$ is an edge of $A_i$ and $A_{i+1}$  for each $i\in \{1,\ldots, n-1\}$. Such a sequence is called a path connecting the cells $A$ and $B$.
\begin{Definition}
A collection of cells $\P$ is called a polyomino if any two cells of $\P$ are connected.
\end{Definition}
\begin{Definition}
A polyomino $\P$ is called row (respectively column) convex, if for any two cells $A$ and $B$ of $\P$ with left lower corners $a=(i,j)$ and $b=(k,j)$ (respectively $a=(i,j)$ and $b=(i,l)$), the horizontal (respectively vertical) cell interval $[A,B]$ is contained in $\P$. If $\P$ is row and column convex, then $\P$ is called a convex polyomino.
\end{Definition}

In Figure~\ref{fig1} we have two examples of polyominoes. The right side one is a convex polyomino, while the other one is not convex because it is not column convex.

\begin{figure}
  \centering
  \begin{tikzpicture}[domain=-1:12]
        \filldraw[fill=black!5!white, draw=black] (0,0) rectangle (1,1);
        \filldraw[fill=black!5!white, draw=black] (1,0) rectangle (2,1);
        \filldraw[fill=black!5!white, draw=black] (2,0) rectangle (3,1);
        \filldraw[fill=black!5!white, draw=black] (1,1) rectangle (2,2);
        \filldraw[fill=black!5!white, draw=black] (2,1) rectangle (3,2);
        \filldraw[fill=black!5!white, draw=black] (3,1) rectangle (4,2);
        \filldraw[fill=black!5!white, draw=black] (2,2) rectangle (3,3);
        \filldraw[fill=black!5!white, draw=black] (3,2) rectangle (4,3);
        \filldraw[fill=black!5!white, draw=black] (0,3) rectangle (1,4);
        \filldraw[fill=black!5!white, draw=black] (1,3) rectangle (2,4);
        \filldraw[fill=black!5!white, draw=black] (2,3) rectangle (3,4);
        \filldraw[fill=black!5!white, draw=black] (3,3) rectangle (4,4);
        \filldraw (2, -1) circle (0 pt) node[below]{A polyomino};
        \filldraw[fill=black!5!white, draw=black] (8,0) rectangle (9,1);
        \filldraw[fill=black!5!white, draw=black] (7,1) rectangle (8,2);
        \filldraw[fill=black!5!white, draw=black] (8,1) rectangle (9,2);
        \filldraw[fill=black!5!white, draw=black] (9,1) rectangle (10,2);
        \filldraw[fill=black!5!white, draw=black] (8,2) rectangle (9,3);
        \draw[-,black, dashed] (7,0)node[below] {$1$} -- (7,4);
        \draw[-,black, dashed] (8,0)node[below] {$2$} -- (8,4);
        \draw[-,black, dashed] (9,0)node[below] {$3$} -- (9,4);
        \draw[-,black, dashed] (10,0)node[below] {$4$} -- (10,4);
        \draw[-,black, dashed] (7,0)node[left] {$1$} -- (11,0);
        \draw[-,black, dashed] (7,1)node[left] {$2$} -- (11,1);
        \draw[-,black, dashed] (7,2)node[left] {$3$} -- (11,2);
        \draw[-,black, dashed] (7,3)node[left] {$4$} -- (11,3);
        \filldraw (8.5, -1) circle (0 pt) node[below] {A convex polyomino};
  \end{tikzpicture}
  \caption{}\label{fig1}
\end{figure}

Let $\P$ be a convex polyomino and $[a,b]\subset \NN^2$ be the smallest interval which contains $V(P)$. After a shift of coordinates, we may assume that $a=(1,1)$ and $b=(m,n)$ and thus, we say that $\P$ is a convex polyomino on $[m]\times [n]$, where for a positive integer $a$, $[a]$ denotes the set $\{1,\ldots, a\}$. For example, the right sight polyomino of Figure~\ref{fig1} is a convex polyomino on $[4]\times[4]$.

Fix a field $\mathbb{K}$ and a polynomial ring $S=\mathbb{K}[x_{ij}\mid (i,j)\in V(\P)]$. We consider the ideal $I_{\P}\subset S$ generated by all binomials $x_{il}x_{kj}-x_{ij}x_{kl}$ for which $[(i,j),(k,l)]$ is an interval in $\P$. The $\mathbb{K}$-algebra $S/I_{\P}$ is denoted $\mathbb{K}[\P]$ and is called the coordinate ring of $\P$. By \cite[Theorem 2.2]{Q}, $\mathbb{K}[\P]$ is a normal Cohen-Macaulay domain.

Let $\P$ be a convex polyomino on $[m]\times[n]$. The ring $R=\mathbb{K}[x_iy_j\mid (i,j)\in V(\P)]\subset \mathbb{K}[x_1,\ldots, x_m,y_1,\ldots, y_n]$ can be viewed as the edge ring of the bipartite graph $G_{\P}$ with vertex set $V(G_{\P})=X\cup Y$, where $X=\{x_1,\ldots, x_m\}$ and $Y=\{y_1,\ldots, y_n\}$ and edge set $E(G_{\P})=\{\{x_i, y_j\}\mid (i,j)\in V(\P)\}$. In Figure~\ref{fig 3} we displayed the bipartite graph attached to a cell in $\mathbb{N}^2$. According to \cite{Q}, $\mathbb{K}[{\P}]$ can be identified with $\mathbb{K}[G_{\P}]$.

\begin{center}
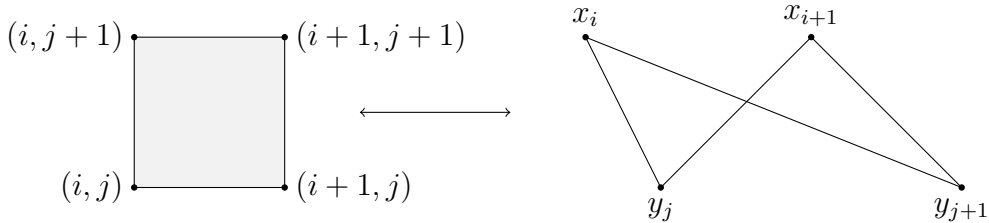
\begin{figure}
  \begin{tikzpicture}[domain=0:12]
        \filldraw[fill=black!5!white, draw=black] (0,0) rectangle (2,2);
        \filldraw (0,0) circle (1pt) node[left] {$(i,j)$};
        \filldraw (2,0) circle (1pt) node[right] {$(i+1,j)$};
        \filldraw (0,2) circle (1pt) node[left] {$(i,j+1)$};
        \filldraw (2,2) circle (1pt) node[right] {$(i+1,j+1)$};
        \draw[<->,black] (3,1) -- (5,1);
        \filldraw (6,2) circle (1pt) node[above] {$x_i$};
        \filldraw (9,2) circle (1pt) node[above] {$x_{i+1}$};
        \filldraw (7,0) circle (1pt) node[below] {$y_j$};
        \filldraw (11,0) circle (1pt) node[below] {$y_{j+1}$};
        \draw[-,black] (6,2) -- (7,0);
        \draw[-,black] (9,2) -- (7,0);
        \draw[-,black] (6,2) -- (11,0);
        \draw[-,black] (9,2) -- (11,0);
    \end{tikzpicture}
  \caption{The bipartite graph attached to a cell in $\mathbb{N}^2$}\label{fig 3}
\end{figure}
\end{center}

\section{Gorenstein convex polyominoes}

Let $\P$ be a convex polyomino on $[m]\times[n]$. We set $X=\{x_1,\ldots, x_m\}$ and $Y=\{y_1,\ldots, y_n\}$ and, if needed, we identify the point $(x_i, y_j)$ in the plane with the vertex $(i,j)\in V(\P)$.

Let $A$ and $B$ be two cells in $\P$. Recall that $A$ and $B$ are connected by a path if there is a sequence of cells in $\P$, $A=A_1, A_2, \ldots, A_{r-1}, A_r=B$, with the property that $A_i\cap A_{i+1}$ is an edge of $A_i$ and $A_{i+1}$, for each $i\in[r-1]$. We denote by $(x_{j_i}, y_{k_i})$ the lower left corner of $A_i$, for all $i\in[r]$. Every path in $\P$ may go in at most four directions which are given below:
\begin{enumerate}
  \item Est if $(x_{j_{i+1}}, y_{k_{i+1}})-(x_{j_i}, y_{k_i})=(1,0)$ for some $i\in[r]$;
  \item West if $(x_{j_{i+1}}, y_{k_{i+1}})-(x_{j_i}, y_{k_i})=(-1,0)$ for some $i\in[r]$;
  \item South if $(x_{j_{i+1}}, y_{k_{i+1}})-(x_{j_i}, y_{k_i})=(0,-1)$ for some $i\in[r]$;
  \item North if $(x_{j_{i+1}}, y_{k_{i+1}})-(x_{j_i}, y_{k_i})=(0,1)$ for some $i\in[r]$.
\end{enumerate}
We say that a path connecting two cells is monotone if it goes only in two directions.
A characterization of convex polyominoes in terms of paths is given by the following Proposition from \cite{CR}.
\begin{Proposition}\label{caract}\cite[Proposition 1]{CR}
A polyomino $\P$ is convex if and only if for every pair of cells there exists a monotone path connecting them and contained in $\P$.
\end{Proposition}

In the next proposition we show that the bipartite graph $G_{\P}$ associated with $\P$ is $2$-connected. Let us first recall the definition of $2$-connectivity.
\begin{Definition}
If $G$ is a finite connected graph on the vertex set $V$, then given a subset $\emptyset \neq W\subset V$, $G_W$ denotes the induced subgraph of $G$ on $W$. We say that $G$ is $2$-connected if $G$ together with $G_{V\setminus \{v\}}$ for all $v\in V$ are connected.
\end{Definition}

\begin{Proposition}
Let $\P$ be a convex polyomino on $[m]\times [n]$. Then the bipartite graph $G_{\P}$ is $2$-connected.
\end{Proposition}
\begin{proof}
  Firstly, we prove that the bipartite graph $G_{\P}$ is connected. For that it is sufficient to choose $x, x'\in \{x_1,\ldots, x_n\}$ and to find a path between them in $G_{\P}$. Let $x,x',y,y'\in V(G)$ such that $(x,y),(x',y')\in V(\P)$.
  Since $\P$ is a convex polyomino, there is a set of vertices of $\P$ \[d=\{(x,y)=(x_{i_0},y_{i_0}),(x_{i_1}, y_{i_1}),\ldots, (x_{i_{r-1}}, y_{i_{r-1}}), (x_{i_r}, y_{i_r})=(x',y')\}\] which are left corners of the cells of a monotone path. Therefore, we have \[(x_{i_{j+1}}, y_{i_{j+1}})-(x_{i_j}, y_{i_j})\in \{(0,-1),(0,1),(1,0),(-1,0)\}\] for any $j\in \{0,\ldots r-1\}$. For this set we denote by
  \[A_d=\{(x,y)\}\cup\{(x_{i_j},y_{i_j})\in d\mid j\in [r-1], (x_{i_{j+1}}, y_{i_{j+1}})-\]\[-(x_{i_{j-1}},y_{i_{j-1}})\in \{(1,1), (-1, -1)\}\}\cup\{(x',y')\}\] the set of the vertices in $d$ where the path changes its direction. We observe that any two consecutive vertices of $A_d$ are either in the same horizontal edge interval or in the same vertical edge interval. Thus, the set $A_d$ determines a path between $x$ and $x'$ in $G_{\P}$.

  In order to complete the proof, we show that for any $k\in [m]$, the graph $G_{\P_V}$ is connected, where $V=V(G_{\P})\setminus\{x_k\}$. Let $G=G_{\P_V}$ and consider $x,x',y,y'\in V(G)$ such that $(x,y),(x',y')\in V(\P)$.

   We consider the set of vertices of $\P$ in a similar way as in the first part of the proof \[d=\{(x,y)=(x_{i_0},y_{i_0}),(x_{i_1}, y_{i_1}),\ldots, (x_{i_{r-1}}, y_{i_{r-1}}), (x_{i_r}, y_{i_r})=(x',y')\}\subset V(\P)\] and  we denote by $A_d$ the set of vertices in $d$ where the path changes its direction.

  If for any $j\in [r-1]$ with $(x_{i_j}, y_{i_j})\in A_d$, we have $x_{i_j}\neq x_k$, then the elements of the set $A_d$ determine a path between $x$ and $x'$ by the argument used above.

  If there is $j\in [r-1]$ such that $(x_k, y_{i_j})\in A_d$, then we show that there is another set $d'$ of vertices of $\P$ for which $A_{d'}$ has no vertex with $x$-coordinate equal to $x_k$ and we reduce to the first case.

  We can choose only two elements, $j_1, j_2\in [r-1]$ such that $(x_k, y_{i_{j_1}}), (x_k, y_{i_{j_2}})\in A_d$. Since $\P$ is a convex polyomino, we have \[(x_{k-1}, y_{i_j})\in V(\P)\text{ or }(x_{k+1}, y_{i_j})\in V(\P)\text{, }\forall i_j\in\{i_{j_1},i_{j_1}+1,\ldots, i_{j_2}\}.\] This implies that
  \[d'=\{(x,y)=(x_{i_0}, y_{i_0}), (x_{i_1}, y_{i_1}),\ldots, (x_{i_{j_1-1}}, y_{i_{j_1}}), (x_{i_{j_1-1}}, y_{i_{j_1+1}}), \ldots, (x_{i_{j_1-1}}, y_{i_{j_2}}),\]\[ (x_{i_{j_2}}, y_{i_{j_2}}), (x_{i_{j_2+1}}, y_{i_{j_2+1}}), \ldots, (x_{i_r}, y_{i_r})=(x',y')\}\text{ or }\]
  \[d'=\{(x,y)=(x_{i_0}, y_{i_0}), (x_{i_1}, y_{i_1}),\ldots, (x_{i_{j_1}}, y_{i_{j_1}}), (x_{i_{j_1+1}}, y_{i_{j_1}}), (x_{i_{j_1+1}}, y_{i_{j_1+1}}),\]\[ \ldots, (x_{i_{j_1+1}}, y_{i_{j_2}})=(x_{i_{j_2+1}},y_{i_{j_2+1}}), (x_{i_{j_2+2}}, y_{i_{j_2+2}}), \ldots, (x_{i_r}, y_{i_r})=(x',y')\}.\] \end{proof}

For the characterisation of Gorenstein convex polyominoes  we need the following theorem due to Ohsugi and Hibi (\cite{HO}).
\begin{Theorem}\label{basic}\cite[Theorem 2.1]{HO}
  Let $G$ be a bipartite graph on $X\cup Y$ and suppose that $G$ is $2$-connected. Then the edge ring of $G$ is Gorenstein if and only if $x_1\cdots x_my_1\cdots y_n\in \mathbb{K}[G]$ and one has $|N(T)|=|T|+1$ for every subset $T\subset X$ such that $G_{T\cup N(T)}$ is connected and that $G_{(X\cup Y)\setminus(T\cup N(T))}$ is a connected graph with at least one edge.
\end{Theorem}
We recall that $N(T)=\{y\in V(G)\mid \{x,y\}\in E(G) \text{ for some }x\in T\}$ represents the set of the neighbors of the subset $T\subset V(G)$.

Note that $x_1\cdots x_my_1\cdots y_n\in \mathbb{K}[G]$ if and only if $G$ possesses a perfect matching (i.e. there is a set of edges $E\subset E(G)$ with the property that no two of them have a common vertex and $\cup_{\{x,y\}\in E}\{x,y\}=V(G)$). A characterization of the bipartite graph which possesses a perfect matching is given by Villarreal (\cite{V}).

\begin{Theorem}\label{hall}\cite[Theorem 7.1.9]{V}
A bipartite graph $G$ with the set of vertices $V=X\cup Y$ possesses a perfect matching if and only if one has $|N(T)|\geq|T|$ for every independent subset of vertices $T\subset V$.
\end{Theorem}
Note that a subset of vertices of $G$ is called independent if no two of them are adjacent.

\begin{Corollary}\label{perfectmatch}
Let $\P$ be a convex polyomino on $[m]\times[n]$ and $G_{\P}$ its associated bipartite graph. Then $x_1\cdots x_my_1\cdots y_n\in \mathbb{K}[G_{\P}]$ if and only if $|N(T)|\geq |T|$ for every $T\subset X$ or $T\subset Y$.
\end{Corollary}
\begin{proof}
If $x_1\cdots x_my_1\cdots y_n\in \mathbb{K}[G_{\P}]$, then by Theorem~\ref{hall}, we obtain $|N(T)|\geq|T|$, for every independent subset of vertices $T\subset X\cup Y$. Notice that all subsets $T\subset X$ and $U\subset Y$ are independent.

Conversely, we suppose $|N(T)|\geq |T|$ for every $T\subset X$ or $T\subset Y$. Let \[T=\{x_{i_1},\ldots, x_{i_r},y_{j_1},\ldots, y_{j_s}\}\subset X\cup Y\] be an independent set of vertices with $r, s\geq 1$. Then \[|T|=r+s\leq |N(\{x_{i_1},\ldots, x_{i_r}\})|+|N(\{y_{j_1},\ldots, y_{j_s}\})|.\] Since $N(\{x_{i_1},\ldots, x_{i_r}\})\subset Y$ and $N(\{y_{j_1},\ldots, y_{j_s}\})\subset X$, we have \[|N(\{x_{i_1},\ldots, x_{i_r}\})|+|N(\{y_{j_1},\ldots, y_{j_s}\})|=\]\[=|N(\{x_{i_1},\ldots, x_{i_r}\})\cup N(\{y_{j_1},\ldots, y_{j_s}\})|=|N(T)|.\] Thus, $|T|\leq|N(T)|$ and according to Theorem~\ref{hall}, $x_1\cdots x_my_1\cdots y_n\in \mathbb{K}[G_{\P}]$.
\end{proof}
\begin{Remark}
If the bipartite graph $G$ on $X\cup Y$ has a perfect matching, then $m=n$. Indeed, $|X|\leq|N(X)|=|Y|$ and $|Y|\leq|N(Y)|=|X|$.
\end{Remark}
\begin{Definition}
Let $\P$ be a convex polyomino on $[m]\times [n]$ and $T\subset X$. The set $N_Y(T)=\{y\in Y\mid (x,y)\in V(\P)\text{ for some }x\in T\}$ is called a neighbor vertical interval if $N_Y(T)=\{y_a, y_{a+1}, \ldots, y_b\}$ with $a<b$ and for every $i\in \{a,a+1,\ldots, b-1\}$ there exists $x\in T$ such that $[(x,y_i), (x, y_{i+1})]$ is an edge in $\P$.
\end{Definition}
\begin{Remark}
If the subset $T=\{x\}\subset X$ has only one element, then $N_Y(T)$ is a neighbor vertical interval. Indeed, let $y_{i_1}, y_{i_2}\in N_Y(x)$ with $i_1<i_2$. Since $\P$ is a convex polyomino, there exists a monotone path between the cells containing $(x, y_{i_1})$ and $(x, y_{i_2})$ as corners. We display the possible monotone paths between two cells in Figure~\ref{fig4}. Then we have $[(x,y_{i_1}),(x,y_{i_2})]\subset \P$.
\begin{center}
  \begin{figure}
    \centering
      \begin{tikzpicture}[domain=0:12]
        \filldraw[fill=black!5!white, draw=black] (3,0) rectangle (3.5,0.5);
        \filldraw[fill=black!5!white, draw=black] (3,0.5) rectangle (3.5,1);
        \filldraw[fill=black!5!white, draw=black] (3,1) rectangle (3.5,1.5);
        \filldraw[fill=black!5!white, draw=black] (3,1.5) rectangle (3.5,2);
        \filldraw (3.5,0) circle (0pt) node[below] {$x$};
        \draw[-,black, dashed] (3,2) -- (2,2) -- (2,1.5) -- (1.5,1.5) -- (1.5,0.5) -- (2,0.5) -- (2,0) -- (3,0);
        \draw[-,black, dashed] (3.5,2) -- (4,2) -- (4,1.5) -- (4.5,1.5) -- (4.5,0.5) -- (4,0.5) -- (4,0) -- (3,0);
        \draw[-,black, dashed] (3.25,0.25) -- (3.25,1.75);
        \filldraw[fill=black!5!white, draw=black] (6,0) rectangle (6.5,0.5);
        \filldraw[fill=black!5!white, draw=black] (6,0.5) rectangle (6.5,1);
        \filldraw[fill=black!5!white, draw=black] (6,1) rectangle (6.5,1.5);
        \filldraw[fill=black!5!white, draw=black] (6,1.5) rectangle (6.5,2);
        \filldraw[fill=black!5!white, draw=black] (6.5,1) rectangle (7,1.5);
        \filldraw[fill=black!5!white, draw=black] (6.5,1.5) rectangle (7,2);
        \filldraw[fill=black!5!white, draw=black] (6.5,2) rectangle (7,2.5);
        \filldraw[fill=black!5!white, draw=black] (6.5,2.5) rectangle (7,3);
        \filldraw (6.5,0) circle (0pt) node[below] {$x$};
        \draw[-,black, dashed] (6,2) -- (5.5,2) -- (5.5,1.5) -- (5,1.5) -- (5,0.5) -- (5.5,0.5) -- (5.5,0) -- (6,0);
        \draw[-,black, dashed] (7,3) -- (7.5,3) -- (7.5,2.5) -- (8,2.5) -- (8,1.5) -- (7.5,1.5) -- (7.5,1) -- (7,1);
        \draw[-,black, dashed] (6.25,0.25) -- (6.25,1.25) -- (6.75,1.25) -- (6.75,2.75);
        \filldraw (3.5,0) circle (1.5pt);
        \filldraw (3.5,2) circle (1.5pt);
        \filldraw (6.5,0) circle (1.5pt);
        \filldraw (6.5,3) circle (1.5pt);
    \end{tikzpicture}
    \caption{Possible monotone paths}\label{fig4}
  \end{figure}
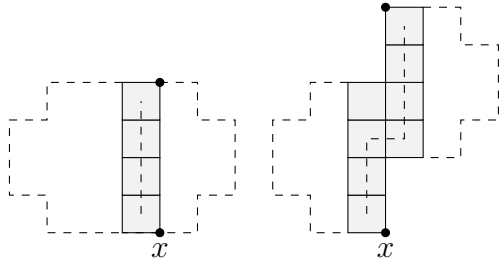
\end{center}
\end{Remark}
\begin{Example}
In the polyomino of Figure~\ref{vertical}, let $T_1=\{x_1, x_4\}$ and $T_2=\{x_1, x_2\}$. Then $N_Y(T_1)=\{y_1, y_2, y_3, y_4\}$ and $N_Y(T_2)=\{y_1, y_2, y_3, y_4\}$.
\begin{center}
\begin{figure}
  \begin{tikzpicture}[domain=0:12]
        \filldraw[fill=black!5!white, draw=black] (0,2) rectangle (1,3);
        \filldraw[fill=black!5!white, draw=black] (1,0) rectangle (2,1);
        \filldraw[fill=black!5!white, draw=black] (1,1) rectangle (2,2);
        \filldraw[fill=black!5!white, draw=black] (1,2) rectangle (2,3);
        \filldraw[fill=black!5!white, draw=black] (2,0) rectangle (3,1);
        \filldraw (4,3) circle (1pt) node[above] {$x_1$};
        \filldraw (7,3) circle (1pt) node[above] {$x_4$};
        \filldraw (4,1) circle (1pt) node[below] {$y_3$};
        \filldraw (5,1) circle (1pt) node[below] {$y_4$};
        \filldraw (6,1) circle (1pt) node[below] {$y_1$};
        \filldraw (7,1) circle (1pt) node[below] {$y_2$};
        \draw[-,black] (4,3) -- (4,1);
        \draw[-,black] (4,3) -- (5,1);
        \draw[-,black] (7,3) -- (6,1);
        \draw[-,black] (7,3) -- (7,1);
        \filldraw (6,0) node[left] {$G_{T_1\cup N(T_1)}$};
        \filldraw (9,3) circle (1pt) node[above] {$x_1$};
        \filldraw (10,3) circle (1pt) node[above] {$x_2$};
        \filldraw (9,1) circle (1pt) node[below] {$y_1$};
        \filldraw (10,1) circle (1pt) node[below] {$y_2$};
        \filldraw (11,1) circle (1pt) node[below] {$y_3$};
        \filldraw (12,1) circle (1pt) node[below] {$y_4$};
        \draw[-,black] (9,3) -- (11,1);
        \draw[-,black] (9,3) -- (12,1);
        \draw[-,black] (10,3) -- (9,1);
        \draw[-,black] (10,3) -- (10,1);
        \draw[-,black] (10,3) -- (11,1);
        \draw[-,black] (10,3) -- (12,1);
        \filldraw (11,0) node[left] {$G_{T_2\cup N(T_2)}$};
    \end{tikzpicture}
    \caption{}\label{vertical}
\end{figure}
\end{center}

We observe that $G_{T_1\cup N(T_1)}$ is not connected, while $G_{T_2\cup N(T_2)}$ is connected.
\end{Example}

\begin{Remark}
In the previous example, we notice that $N_Y(T_1)$ and $N_Y(T_2)$ coincide as sets, but $N_Y(T_2)$ is a neighbor vertical interval, while $N_Y(T_1)$ is not.
\end{Remark}

\begin{Lemma}\label{lemma1}
Let $\P$ be a convex polyomino and $G:=G_{\P}$ its associated bipartite graph. Then for each $\emptyset\neq T\subsetneq X$, the following conditions are equivalent:
\begin{enumerate}
  \item $N_Y(T)$ is a neighbor vertical interval.
  \item $G_{T\cup N(T)}$ is a connected graph.
\end{enumerate}
\end{Lemma}
\begin{proof}
For \emph{(1)$\Rightarrow $(2)}, it is sufficient to choose $x, z\in T$ and to find a path $d$ between them in $G_{T\cup N(T)}$.
Without loss of generality, we may choose $y_s\in N_Y(x)$ and $y_t\in N_Y(z)$ with $s<t$. Then by hypothesis, $\{y_{s},y_{s+1}, \ldots, y_{t}\}\subset N_Y(T)$ and there exist $x_{i_s}, x_{i_{s+1}}, \ldots, x_{i_{t-1}}\in T$ such that $[(x_{i_j},y_j),(x_{i_j}, y_{j+1})]$ is an edge in $\P$, for $j\in\{s,s+1, \ldots, t-1\}$. Thus, we have \[(x, y_{s}), (x_{i_s}, y_{s}), (x_{i_{s}}, y_{s+1}),(x_{i_{s+1}}, y_{s+1}), \dots, (x_{i_{t-1}}, y_{t-1}), (x_{i_{t-1}}, y_{t}) (z, y_{t})\in V(\P).\] So the path between $x$ and $z$ in $G_{\P}$ is \[d=\{\{x, y_{s}\}, \{y_{s}, x_{i_s}\}, \{x_{i_s}, y_{s+1}\},\{y_{s+1}, x_{i_{s+1}}\}, \ldots, \{x_{i_{t-1}}, y_{t}\}, \{y_{t}, z\}\}.\]

For \emph{(2)$\Rightarrow $(1)}, we consider $N_Y(T)=\{y_{i_1}, \dots, y_{i_s}\mid i_1<i_2<\cdots<i_s\}$ and we prove that for every $j\in[s-1]$, there exists $x_k\in T$ such that $$[(x_k,y_{i_j}),(x_k, y_{i_{j+1}})]$$ is an edge in $\P$. In particular, it also follows that $i_{j+1}=i_j+1$ for each $j\in[s-1]$, which will end the proof.

Let $j\in[s-1]$. Since $G_{T\cup N(T)}$ is a connected graph, there is a path between $y_{i_j}$ and $y_{i_{j+1}}$ in $G_{T\cup N(T)}$. In other words, there are $x_{k_1}, \ldots, x_{k_{r-1}}\in T$ and $y_{l_1}, \ldots, y_{l_{r-2}}\in N_Y(T)$ such that $$(x_{k_1}, y_{i_j}), (x_{k_1}, y_{l_1}), (x_{k_2}, y_{l_1}), (x_{k_2},y_{l_2}), \ldots, (x_{k_{r-1}}, y_{l_{r-2}}), (x_{k_{r-1}}, y_{i_{j+1}})\in V(\P).$$

If we have $a\in[r-2]$ such that $l_a<i_j<i_{j+1}<l_{a+1}$, then $$[(x_{k_{a+1}}, y_{i_j}), (x_{k_{a+1}}, y_{i_{j+1}})]$$ is an edge interval in $\P$ because $(x_{k_{a+1}}, y_{l_a}), (x_{k_{a+1}}, y_{l_{a+1}})\in V(\P)$ and $N_Y(x_{k_{a+1}})$ is a neighbor vertical interval. Moreover, $y_{i_j}, y_{i_j+1},\ldots, y_{i_{j+1}}\in N_Y(x_{k_{a+1}})\subset N_Y(T)$. Thus, $i_{j+1}=i_j+1$ and $[(x_{k_{a+1}}, y_{i_j}), (x_{k_{a+1}}, y_{i_{j+1}})]$ is an edge in $\P$.

If for all $a\in[r-2]$, $l_a<i_j<i_{j+1}$, then $$[(x_{k_{r-1}}, y_{i_j}), (x_{k_{r-1}}, y_{i_{j+1}})]$$  is an edge interval in $\P$, since $(x_{k_{r-1}}, y_{l_{r-2}}), (x_{k_{r-1}}, y_{i_{j+1}})\in V(\P)$ and $N_Y(x_{k_{r-1}})$ is a neighbor vertical interval. So $y_{i_j}, y_{i_j+1},\ldots, y_{i_{j+1}}\in N_Y(x_{k_{a+1}})\subset N_Y(T)$ and $[(x_{k_{r-1}}, y_{i_j}), (x_{k_{r-1}}, y_{i_{j+1}})]$ is an edge in $\P$. We proceed in a similar way in the case that for all $a\in[r-2]$ we have $i_j<i_{j+1}<l_a$.
\end{proof}

\begin{Definition}
Let $\P$ be a convex polyomino on $[m]\times [n]$ and $U\subset Y$. The set $N_X(U)=\{x\in X\mid (x,y)\in V(\P)\text{ for some }y\in U\}$ is called a neighbor horizontal interval if $N_X(U)=\{x_a, x_{a+1}, \ldots, x_b\}$ with $a<b$ and for every $i\in \{a,a+1,\ldots, b-1\}$ there exists $y\in U$ such that $[(x_i,y), (x_{i+1}, y)]$ is an edge in $\P$.
\end{Definition}

\begin{Remark}
In the case that the subset $U=\{y\}\subset Y$ has only one element, we consider $x_{i_1}, x_{i_2}\in N_X(U)$ with $i_1<i_2$. Since $\P$ is a convex polyomino, by Proposition~\ref{caract}, there is a monotone path between the cells containing $(x_{i_1},y)$ and $(x_{i_2},y)$ as corners. We display the possible monotone paths between two cells in Figure~\ref{fig5}. Then we have $[(x_{i_1}, y),(x_{i_2},y)]\subset \P$.

\begin{center}
  \begin{figure}
    \centering
      \begin{tikzpicture}[domain=0:12]
        \filldraw[fill=black!5!white, draw=black] (1,2) rectangle (1.5,2.5);
        \filldraw[fill=black!5!white, draw=black] (1.5,2) rectangle (2,2.5);
        \filldraw[fill=black!5!white, draw=black] (2,2) rectangle (2.5,2.5);
        \filldraw[fill=black!5!white, draw=black] (2.5,2) rectangle (3,2.5);
        \filldraw (1,2) circle (0pt) node[left] {$y$};
        \draw[-,black, dashed] (1,2) -- (1,1.5) -- (1.5,1.5) -- (1.5,1) -- (2.5,1) -- (2.5,1.5) -- (3,1.5) -- (3,2);
        \draw[-,black, dashed] (1,2.5) -- (1,3) -- (1.5,3) -- (1.5,3.5) -- (2.5,3.5) -- (2.5,3) -- (3,3) -- (3,2.5);
        \draw[-,black, dashed] (1.25,2.25) -- (2.75,2.25);
        \filldraw[fill=black!5!white, draw=black] (5,2) rectangle (5.5,2.5);
        \filldraw[fill=black!5!white, draw=black] (5.5,2) rectangle (6,2.5);
        \filldraw[fill=black!5!white, draw=black] (6,2) rectangle (6.5,2.5);
        \filldraw[fill=black!5!white, draw=black] (6.5,2) rectangle (7,2.5);
        \filldraw[fill=black!5!white, draw=black] (7,1.5) rectangle (7.5,2);
        \filldraw[fill=black!5!white, draw=black] (7.5,1.5) rectangle (8,2);
        \filldraw[fill=black!5!white, draw=black] (6,1.5) rectangle (6.5,2);
        \filldraw[fill=black!5!white, draw=black] (6.5,1.5) rectangle (7,2);
        \filldraw (5,2) circle (0pt) node[left] {$y$};
        \draw[-,black, dashed] (5,2.5) -- (5,3) -- (5.5,3) -- (5.5,3.5) -- (6.5,3.5) -- (6.5,3) -- (7,3) -- (7,2.5);
        \draw[-,black, dashed] (6,1.5) -- (6,1) -- (6.5,1) -- (6.5,0.5) -- (7.5,0.5) -- (7.5,1) -- (8,1) -- (8,1.5);
        \draw[-,black, dashed] (5.25,2.25) -- (6.75,2.25) -- (6.75,1.75) -- (7.75,1.75);
        \filldraw (1,2) circle (1.5pt);
        \filldraw (3,2) circle (1.5pt);
        \filldraw (5,2) circle (1.5pt);
        \filldraw (8,2) circle (1.5pt);
    \end{tikzpicture}
    \caption{Possible monotone paths}\label{fig5}
  \end{figure}
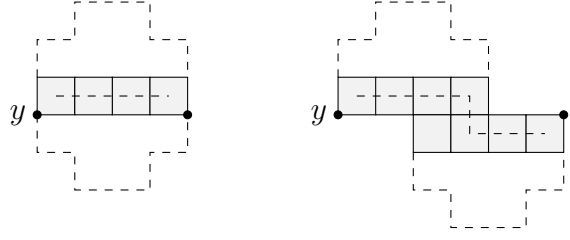
\end{center}
\end{Remark}
\begin{Example}
In the polyomino of Figure~\ref{fig6}, let $U_1=\{y_2, y_3\}$ and $U_2=\{y_1, y_5\}$. We observe that $N_X(U_1)=\{x_1, x_2, x_3, x_4, x_5\}$ is a neighbor horizontal interval, while $N_X(U_2)=\{x_1, x_2, x_3, x_4\}$ is not.

\begin{center}
\begin{figure}
  \begin{tikzpicture}[domain=0:12]
        \filldraw[fill=black!5!white, draw=black] (0,1) rectangle (1,2);
        \filldraw[fill=black!5!white, draw=black] (0,2) rectangle (1,3);
        \filldraw[fill=black!5!white, draw=black] (0,3) rectangle (1,4);
        \filldraw[fill=black!5!white, draw=black] (1,1) rectangle (2,2);
        \filldraw[fill=black!5!white, draw=black] (2,0) rectangle (3,1);
        \filldraw[fill=black!5!white, draw=black] (2,1) rectangle (3,2);
        \filldraw[fill=black!5!white, draw=black] (3,1) rectangle (4,2);
    \end{tikzpicture}
\caption{}\label{fig6}
\end{figure}
\end{center}
\end{Example}

\begin{Lemma}\label{lemma2}
If $\P$ is a convex polyomino, then for each $\emptyset\neq T\subset X$, $N_Y(x)\nsubseteq N_Y(T)\text{, }\forall\text{ }x\in X\setminus T$ if and only if $N_X(Y\setminus N_Y(T))=X\setminus T$.
\end{Lemma}
\begin{proof}
In the implication from left to right, we prove that $N_X(Y\setminus N_Y(T))=X\setminus T$. Let $x\in N_X(Y\setminus N_Y(T))$. Then there exists $y\in Y\setminus N_Y(T)$ such that $(x,y)\in V(\P)$. If $x\in T$, then $y\in N_Y(x)\subset N_Y(T)$, which is false. Thus, $x\in X\setminus T$ and $N_X(Y\setminus N_Y(T))\subset X\setminus T$.

If $x\in X\setminus T$, then by hypothesis, we obtain $N_Y(x)\nsubseteq N_Y(T)$. So there exists $y\in N_Y(x)\setminus N_Y(T)$. Hence, $y\in Y\setminus N_Y(T)$ and $x\in N_X(y)\subset N_X(Y\setminus N_Y(T))$. In other words,
$X\setminus T\subset N_X(Y\setminus N_Y(T))$.

Conversely, let $x\in X\setminus T$. Since $N_X(Y\setminus N_Y(T))=X\setminus T$, $x\in N_X(Y\setminus N_Y(T))$ and there exists $y\in Y\setminus N_Y(T)$ such that $(x,y)\in V(\P)$. This is equivalent to say that $y\in N_Y(x)\setminus N_Y(T)$ and $N_Y(x)\nsubseteq N_Y(T)$.
\end{proof}
\begin{Example}
In Figure~\ref{fig6}, let $T=\{x_4,x_5\}$. We observe that $N_X(Y\setminus N_Y(T))=N_X(\{y_4, y_5\})=\{x_1, x_2\}\neq \{x_1, x_2, x_3\}=X\setminus T$. On the other hand, $x_3\notin T$ and $N_Y(x_3)=N_Y(T)$.
\end{Example}

\begin{Lemma}\label{lemma3}
Let $\P$ be a convex polyomino and $G:=G_{\P}$ its associated bipartite graph. Then for each $\emptyset\neq T\subsetneq X$, the following conditions are equivalent:
\begin{enumerate}
  \item $N_X(Y\setminus N_Y(T))=X\setminus T$ is a neighbor horizontal interval.
  \item $G_{(X\cup Y)\setminus(T\cup N_Y(T))}$ is a connected graph with at least one edge.
\end{enumerate}
\end{Lemma}
\begin{proof}
Let $T$ be a subset in $X$ which satisfies the conditions given in \emph{(1)}.
By Lemma~\ref{lemma2} and the fact that $T\subsetneq X$, there is $x\in X\setminus T$ with $N_Y(x)\nsubseteq N_Y(T)$. In other words, there are $x\in X\setminus T$ and $y\in Y\setminus N_Y(T)$ such that $(x,y)\in V(\P)$. This is equivalent to saying that $\{x,y\}$ is an edge in $G_{(X\cup Y)\setminus(T\cup N_Y(T))}$.

For the connectivity of the graph $G_{(X\cup Y)\setminus(T\cup N_Y(T))}$, it is sufficient to choose $y,z\in Y\setminus N_Y(T)$ and to find a path between them in $G_{(X\cup Y)\setminus(T\cup N_Y(T))}$. Without loss of generality, we consider $x_s\in N_X(y)$ and $x_t\in N_X(z)$ with $s<t$. Since $N_X(Y\setminus N_Y(T))$ is a neighbor horizontal interval, $\{x_s, x_{s+1},\ldots, x_t\}\subset N_X(Y\setminus N_Y(T))$ and there exist $y_{i_s}, y_{i_{s+1}}, \ldots, y_{i_{t-1}}\in Y\setminus N_Y(T)$ such that $[(x_j, y_{i_j}),(x_{j+1}, y_{i_j})]$ is an edge in $\P$, for each $j\in \{s, s+1,\ldots, t-1\}$. It follows that \[(x_s, y), (x_s, y_{i_s}), (x_{s+1}, y_{i_s}), (x_{s+1}, y_{i_{s+1}}), \ldots, (x_{t-1}, y_{i_{t-1}}), (x_t, y_{i_{t-1}}), (x_t, z)\in V(\P).\] In other words, the path between $y$ and $z$ is \[d=\{\{y, x_s\},\{x_s, y_{i_s}\}, \{y_{i_s}, x_{s+1}\}, \{x_{s+1}, y_{i_{s+1}}\}, \ldots, \{y_{i_{t-1}}, x_t\}, \{x_t, z\}\}.\]

Conversely, we suppose that $X\setminus T \neq N_X(Y\setminus N_Y(T))$. By Lemma~\ref{lemma2}, there is $x\in X\setminus T$ with the property that $N_Y(x)\subset N_Y(T)$. So $x$ represents an isolated vertex in $G_{(X\cup Y)\setminus(T\cup N_Y(T))}$, which is false.

Now, we consider $N_X(Y\setminus N_Y(T))=\{x_{i_1}, \ldots, x_{i_s}\mid i_1<\cdots<i_s\}$ and we prove that for every $j\in[s-1]$ there exists $y_k\in Y\setminus N_Y(T)$ such that $[(x_{i_j}, y_k),(x_{i_{j+1}}, y_k)]$ is an edge in $\P$.

Let $j\in[s-1]$. Since $G_{(X\cup Y)\setminus(T\cup N_Y(T))}$ is a connected graph, there is a path between $x_{i_j}$ and $x_{i_{j+1}}$ in $G_{(X\cup Y)\setminus(T\cup N_Y(T))}$. Thus, there exist $x_{l_1}, \ldots, x_{l_{r-2}}\in N_X(Y\setminus N_Y(T))$ and $y_{k_1}, \ldots, y_{k_{r-1}}\in Y\setminus N_Y(T)$ such that \[(x_{i_j}, y_{k_1}), (x_{l_1}, y_{k_1}), (x_{l_1}, y_{k_2}), \ldots, (x_{l_{r-2}}, y_{k_{r-1}}), (x_{i_{j+1}}, y_{k_{r-1}})\in V(\P).\]

Now we argue similarly to the proof of Lemma~\ref{lemma1}.

If we find $a\in[r-2]$ such that $l_a<i_j<i_{j+1}<l_{a+1}$, then \[[(x_{i_j}, y_{k_{a+1}}),(x_{i_{j+1}}, y_{k_{a+1}})]\] is an edge interval in $\P$ because $(x_{l_a}, y_{k_{a+1}}), (x_{l_{a+1}}, y_{k_{a+1}})\in V(\P)$ and $N_X(y_{k_{a+1}})$ is a neighbor horizontal interval. Moreover $x_{i_j}, x_{i_j+1}, \ldots, x_{i_{j+1}}\in N_X(y_{k_{a+1}})\subset N_X(Y\setminus N_Y(T))$ and so, $i_{j+1}=i_j+1$ and $[(x_{i_j}, y_{k_{a+1}}),(x_{i_{j+1}}, y_{k_{a+1}})]$ is an edge in $\P$.

If for all $a\in[r-2]$, $i_j<i_{j+1}<l_a$, then $$[(x_{i_j}, y_{k_1}), (x_{i_{j+1}}, y_{k_1})]$$ is an edge interval in $\P$ because $(x_{i_j}, y_{k_1}),(x_{l_1}, y_{k_1})\in V(\P)$ and $N_X(y_{k_1})$ is a neighbor horizontal interval. Thus, $x_{i_j}, x_{i_j+1}, \ldots, x_{i_{j+1}}\in N_X(y_{k_1})\subset N_X(Y\setminus N_Y(T))$ and $[(x_{i_j}, y_{k_1}), (x_{i_{j+1}}, y_{k_1})]$ is an edge in $\P$. We proceed in a similar way in the case that for all $a\in[r-2]$ we have $l_a<i_j<i_{j+1}$.
\end{proof}

\begin{Example}
In Figure~\ref{fig6}, let $T_1=\{x_5\}$ and $T_2=\{x_1, x_2, x_3\}$. We observe that $G_{(X\cup Y)\setminus(T_1\cup N_Y(T_1))}$ is not connected because $N_X(Y\setminus N_Y(T_1))=X\setminus T_1$ is not a neighbor horizontal interval (Figure~\ref{fig7}). The graph $G_{(X\cup Y)\setminus(T_2\cup N_Y(T_2))}$ is represented by only two isolated vertex $x_4$ and $x_5$.

\begin{figure}
    \centering
      \begin{tikzpicture}[domain=0:12]
        \filldraw (0,3) circle (1pt) node[above] {$x_1$};
        \filldraw (1,3) circle (1pt) node[above] {$x_2$};
        \filldraw (2,3) circle (1pt) node[above] {$x_3$};
        \filldraw (3,3) circle (1pt) node[above] {$x_4$};
        \filldraw (3,1) circle (1pt) node[below] {$y_1$};
        \filldraw (0,1) circle (1pt) node[below] {$y_4$};
        \filldraw (1,1) circle (1pt) node[below] {$y_5$};
        \draw[-,black] (2,3) -- (3,1) -- (3,3);
        \draw[-,black] (0,3) -- (0,1) -- (1,3) -- (1,1) -- (0,3);
    \end{tikzpicture}
    \caption{$G_{(X\cup Y)\setminus(T_1\cup N_Y(T_1))}$}\label{fig7}
\end{figure}
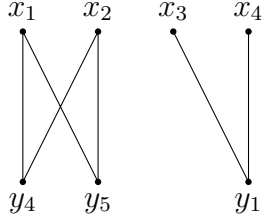
\end{Example}

Let $\P$ be a convex polyomino. Since the coordinate ring of $\P$ can be viewed as an edge ring of a bipartite graph, by applying Theorem~\ref{basic}, Corollary~\ref{perfectmatch}, Lemma~\ref{lemma1} and Lemma~\ref{lemma3} we get the following result.

\begin{Theorem}\label{mainresult}
Let $\P$ be a convex polyomino on $[m]\times[n]$ and $G:=G_{\P}$ its associated bipartite graph.

Then $\mathbb{K}[\P]$ is Gorenstein if and only if the following conditions are fulfilled:
\begin{enumerate}
\item $|U|\leq |N_X(U)|$, for every $U\subset Y$ and $|T|\leq|N_Y(T)|$ for every $T\subset X$;
\item For every $\emptyset\neq T\subsetneq X$ with the properties
    \begin{enumerate}
    \item $N_Y(T)$ is a neighbor vertical interval,
    \item $N_X(Y\setminus N_Y(T))=X\setminus T$ is a neighbor horizontal interval,
  \end{enumerate} one has $|N_Y(T)|=|T|+1$.
\end{enumerate}
\end{Theorem}

\begin{Examples}
Let $\P_1$ be the polyomino of Figure~\ref{fig8}.
\begin{enumerate}
  \item Let $T=\{x_4, x_5, x_6\}$. $T$ satisfies properties $(a),(b)$. Since $|N_Y(T)|=3\neq 4=|T|+1$, $\P_1$ is not a Gorenstein polyomino.
  \item For $T=\{x_1, x_4, x_5, x_6\}$, only the property $(b)$ is fulfiled.
  \item For $T=\{x_4\}$, we have property $(a)$ and $N_X(Y\setminus N_Y(T))$ is a neighbor horizontal interval, but $X\setminus T\neq N_X(Y\setminus N_Y(T))$.
  \item For $T=\{x_6\}$, we have property $(a)$, but $N_X(Y\setminus N_Y(T))=X\setminus T$ is not a neighbor horizontal interval.
\end{enumerate}

\begin{figure}
  \centering
  \begin{tikzpicture}[domain=1:6]
        \filldraw[fill=black!5!white, draw=black] (1,4) rectangle (2,5);
        \filldraw[fill=black!5!white, draw=black] (2,2) rectangle (3,3);
        \filldraw[fill=black!5!white, draw=black] (2,3) rectangle (3,4);
        \filldraw[fill=black!5!white, draw=black] (2,4) rectangle (3,5);
        \filldraw[fill=black!5!white, draw=black] (2,5) rectangle (3,6);
        \filldraw[fill=black!5!white, draw=black] (3,2) rectangle (4,3);
        \filldraw[fill=black!5!white, draw=black] (4,2) rectangle (5,3);
        \filldraw[fill=black!5!white, draw=black] (4,1) rectangle (5,2);
        \filldraw[fill=black!5!white, draw=black] (5,2) rectangle (6,3);
\end{tikzpicture}
\caption{}\label{fig8}
\end{figure}

The polyomino $\P_2$ of Figure~\ref{fig9} is Gorenstein, because $x_1y_1\cdot x_2y_2\cdot x_3y_4\cdot x_4y_3\in \mathbb{K}[\P_2]$ and for each $T$ which satisfies the properties $(a),(b)$, one has $|N_Y(T)|=|T|+1$. In this case, we need to check the conditions of the Theorem only for two sets:
\begin{enumerate}
  \item $T=\{x_4\}$ with $N_Y(T)=\{y_2, y_3\}$;
  \item $T=\{x_1, x_4\}$ with $N_Y(T)=\{y_1, y_2, y_3\}$.
\end{enumerate}
\begin{figure}
  \centering
  \begin{tikzpicture}[domain=1:4]
        \filldraw[fill=black!5!white, draw=black] (1,1) rectangle (2,2);
        \filldraw[fill=black!5!white, draw=black] (1,2) rectangle (2,3);
        \filldraw[fill=black!5!white, draw=black] (2,1) rectangle (3,2);
        \filldraw[fill=black!5!white, draw=black] (2,2) rectangle (3,3);
        \filldraw[fill=black!5!white, draw=black] (2,3) rectangle (3,4);
        \filldraw[fill=black!5!white, draw=black] (3,2) rectangle (4,3);
\end{tikzpicture}
  \caption{}\label{fig9}
\end{figure}
\end{Examples}

\section{Gorenstein stack polyominoes}

We consider $\P$ to be a polyomino and we may assume that $[(1,1), (m,n)]$ is the smallest interval containing $V(\P)$. Then $\P$ is called a stack polyomino (Figure~\ref{stack}), if it is a convex polyomino and for $i\in [m-1]$, the cells $[(i, 1), (i+1, 2)]$ belong to $\P$.

\begin{figure}
  \centering
  \begin{tikzpicture}[domain=0:9]
        \filldraw[fill=black!5!white, draw=black] (0,0) rectangle (1,1) ;
        \filldraw[fill=black!5!white, draw=black] (1,0) rectangle (2,1) ;
        \filldraw[fill=black!5!white, draw=black] (2,0) rectangle (3,1) ;
        \filldraw[fill=black!5!white, draw=black] (0,1) rectangle (1,2) ;
        \filldraw[fill=black!5!white, draw=black] (1,1) rectangle (2,2) ;
        \filldraw[fill=black!5!white, draw=black] (1,2) rectangle (2,3) ;
        \filldraw[fill=black!5!white, draw=black] (5,0) rectangle (6,1) ;
        \filldraw[fill=black!5!white, draw=black] (6,0) rectangle (7,1) ;
        \filldraw[fill=black!5!white, draw=black] (7,0) rectangle (8,1) ;
        \filldraw[fill=black!5!white, draw=black] (8,0) rectangle (9,1) ;
        \filldraw[fill=black!5!white, draw=black] (5,1) rectangle (6,2) ;
  \end{tikzpicture}
  \caption{Stack polyominoes}\label{stack}
\end{figure}
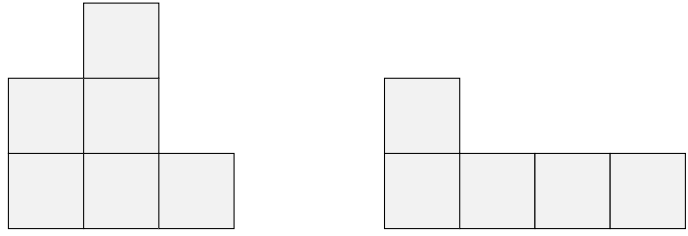

\begin{Remark}\label{rem}
If $\P$ is a stack polyomino, then for every $x\in X$ we have $\{y_1, y_2\}\subset N_Y(x)$.

Let $T\neq \emptyset$ be a subset in $X$ and $y_j\in N_Y(T)\setminus \{y_1, y_2\}$. So there exists $x_k\in T$ such that $y_j\in N_Y(x_k)$. Since $N_Y(x_k)$ is a neighbor vertical interval, \[\{y_1, y_2, \ldots, y_{j-1}, y_j\}\subset N_Y(x_k)\subset N_Y(T).\] In other words, $N_Y(T)=\{y_1, y_2, \ldots, y_s\}$ is a neighbor vertical interval for all $\emptyset\neq T\subsetneq X$, where $s=\max\{i\in [n]\mid (x_k, y_i)\in V(\P)\text{ for some }x_k\in T\}$.
\end{Remark}
\begin{Lemma}\label{lemma0}
Let $\P$ be a stack polyomino on $[n]\times [n]$. If $x_1\cdots x_ny_1\cdots y_n\notin\mathbb{K}[\P]$ then there is a subset $T\subset X$ for which the following conditions hold:
  \begin{enumerate}
    \item $Y\setminus N_Y(T)\neq \emptyset$ and
    \item  $(\forall)\text{ }x\in X\setminus T$, $\max\{j\in [n]\mid y_j\in N_Y(x)\}>\max\{j\in [n]\mid y_j\in N_Y(T)\}$.
  \end{enumerate}
\end{Lemma}
\begin{proof}

We suppose that $x_1\cdots x_ny_1\cdots y_n\notin\mathbb{K}[\P]$. By Corollary~\ref{perfectmatch}, we find $I\subset X$ with $|I|>|N_Y(I)|$ or $J\subset Y$ with $|J|>|N_X(J)|$.

In the case that $I\subset X$ and $|I|>|N_Y(I)|$, we consider $$T=I\cup \{x\in X\mid N_Y(x)\subset N_Y(I)\}.$$ We check conditions \emph{(1)} and \emph{(2)} for the set $T$. Since $\P$ is a stack polyomino, $N_Y(T)=\{y_1, y_2,\ldots, y_s\}$ for some $s\leq n$. If $N_Y(T)=Y$, then $|N_Y(T)|=|Y|=n\geq|I|$, which is false. So, $Y\setminus N_Y(T)\neq \emptyset$. Let $x\in X\setminus T$. It follows that $N_Y(x)\nsubseteq N_Y(I)=N_Y(T)$. Thus, there is $l>s$ such that $y_l\in N_Y(x)\setminus N_Y(T)$ and condition \emph{(2)} holds.

If there exists $J\subset Y$ with $|J|>|N_X(J)|$, then we set $$T=X\setminus N_X(J).$$ We check conditions \emph{(1)} and \emph{(2)} for the set $T$. For the proof of the first condition, it is sufficient to show that $J\subset Y\setminus N_Y(T)$. Let $y\in J$. In the case that $y\in N_Y(T)$, there is $x\in T\cap N_X(y)$. Since $y\in J$, we get $x\in N_X(y)\subset N_X(J)=X\setminus T$, which is impossible. Thus, $\emptyset\neq J\subset Y\setminus N_Y(T)$. It follows that $X\setminus T=N_X(J)\subset N_X(Y\setminus N_Y(T))$. For each $y\in Y\setminus N_Y(T)$, we have $N_X(y)\cap T=\emptyset$. Consequently, $N_X(y)\subset X\setminus T$ and $N_X(Y\setminus N_Y(T))=\cup_{y\in Y\setminus N_Y(T)}N_X(y)\subset X\setminus T$. In other words, $X\setminus T= N_X(Y\setminus N_Y(T))$. By Lemma~\ref{lemma2} and the previous remark for any $x\in X\setminus T$, $N_Y(x)\nsubseteq N(T)$ and we have the second condition.
\end{proof}

%\begin{Examples}
%In the first stack polyomino of Figure~\ref{fig10}, $x_1\cdots x_ny_1\cdots y_n\notin\mathbb{K}[\P]$, while in the second $x_1y_4\cdot x_2y_3\cdot x_3y_2\cdot x_4y_1\in \mathbb{K}[\P]$.
%\begin{figure}
 % \centering
  %\begin{tikzpicture}[domain=1:10]
   %     \filldraw[fill=black!5!white, draw=black] (1,1) rectangle (2,2);
    %    \filldraw[fill=black!5!white, draw=black] (1,2) rectangle (2,3);
     %   \filldraw[fill=black!5!white, draw=black] (1,3) rectangle (2,4);
      %  \filldraw[fill=black!5!white, draw=black] (1,4) rectangle (2,5);
       % \filldraw[fill=black!5!white, draw=black] (2,1) rectangle (3,2);
       % \filldraw[fill=black!5!white, draw=black] (3,1) rectangle (4,2);
       % \filldraw[fill=black!5!white, draw=black] (4,1) rectangle (5,2);
        %\filldraw[fill=black!5!white, draw=black] (7,1) rectangle (8,2);
        %\filldraw[fill=black!5!white, draw=black] (7,2) rectangle (8,3);
        %\filldraw[fill=black!5!white, draw=black] (7,3) rectangle (8,4);
        %\filldraw[fill=black!5!white, draw=black] (8,1) rectangle (9,2);
        %\filldraw[fill=black!5!white, draw=black] (9,1) rectangle (10,2);
        %\filldraw[fill=black, draw=black] (7,4) circle (2pt);
        %\filldraw[fill=black, draw=black] (8,3) circle (2pt);
        %\filldraw[fill=black, draw=black] (9,2) circle (2pt);
        %\filldraw[fill=black, draw=black] (10,1) circle (2pt);
%\end{tikzpicture}
%  \caption{}\label{fig10}
%\end{figure}
%\end{Examples}
As a consequence of Theorem~\ref{mainresult}, we may recover the characterisation of Gorenstein stack polyominoes which was obtained by Qureshi in \cite{Q}.
\begin{Corollary}\label{cor1}
Let $\P$ denote a stack polyomino on $[m]\times [n]$. The following conditions are equivalent:
  \begin{enumerate}
    \item $\mathbb{K}[\P]$ is Gorenstein.
    \item $m=n$ and for every $T\subset X$ with the properties that $Y\setminus N_Y(T)\neq \emptyset$ and $(\forall)\text{ }x\in X\setminus T$, $\max\{j\in [n]\mid y_j\in N_Y(x)\}>\max\{j\in [n]\mid y_j\in N_Y(T)\}$, one has $|N_Y(T)|=|T|+1$.
  \end{enumerate}
\end{Corollary}
\begin{proof}
For \emph{(1)}$\Rightarrow$\emph{(2)}, let $T\neq\emptyset$ be a subset in $X$ such that $Y\setminus N_Y(T)\neq \emptyset$ and $\max\{j\in [n]\mid y_j\in N_Y(x)\}>\max\{j\in [n]\mid y_j\in N_Y(T)\}$, for every $x\in X\setminus T$. By Remark~\ref{rem}, $N_Y(T)$ is a neighbor vertical interval.

By Lemma~\ref{lemma2}, we have $X\setminus T=N_X(Y\setminus N_Y(T))$, since $N_Y(T)=\{y_1,y_2, \ldots, y_s\}$ and $N_Y(x)=\{y_1,y_2,\ldots, y_t\}$ with $t>s$, for every $x\in X\setminus T$. Moreover $Y\setminus N_Y(T)=\{y_{s+1}, y_{s+2}, \ldots, y_n\}\neq\emptyset$ and $y_{s+1}\in N_Y(x)$, $\forall$ $x\in X\setminus T$. So $N_X(Y\setminus N_Y(T))=X\setminus T$ is a neighbor horizontal interval because $N_X(y_{s+1})$ is also a neighbor horizontal interval. By using Theorem~\ref{mainresult} and Corollary~\ref{perfectmatch}, $|N_Y(T)|=|T|+1$ and $x_1\cdots x_my_1\cdots y_n\in \mathbb{K}[\P]$. Thus, we also obtain $m=n$.

For \emph{(2)}$\Rightarrow$\emph{(1)}, we suppose that $m=n$ and $x_1\cdots x_my_1\cdots y_n\notin \mathbb{K}[\P]$. 

By Lemma~\ref{lemma0}, there exists $\emptyset\neq T\subsetneq X$ such that $|T|>|N_Y(T)|$, $Y\setminus N_Y(T)\neq \emptyset$ and $\max\{j\in [n]\mid y_j\in N_Y(x)\}>\max\{j\in [n]\mid y_j\in N_Y(T)\}$, for every $x\in X\setminus T$, which is false. Thus, $x_1\cdots x_ny_1\cdots y_n\in \mathbb{K}[\P]$ and we obtain the first condition of Theorem~\ref{mainresult} by applying Corollary~\ref{perfectmatch}.

Let $\emptyset\neq T\subsetneq X$ such that $N_Y(T)$ is a neighbor vertical interval and $N_X(Y\setminus N_Y(T))=X\setminus T$ is a neighbor horizontal interval. Since $T\subsetneq X$, there exists $x\in X\setminus T$ with $N_Y(x)\nsubseteq N_Y(T)$, by Lemma~\ref{lemma2}. It follows that we find $y\in N_Y(x)\setminus N_Y(T)\subset Y\setminus N_Y(T)$. In other words, $Y\setminus N_Y(T)\neq \emptyset$.

If $x\in X\setminus T$, then $\max\{j\in [n]\mid y_j\in N_Y(x)\}>\max\{j\in [n]\mid y_j\in N_Y(T)\}$ by Lemma~\ref{lemma2} and Remark~\ref{rem}. It implies that $|N_Y(T)|=|T|+1$ and the second condition of Theorem~\ref{mainresult} is fulfilled. Hence, $\mathbb{K}[\P]$ is Gorenstein.
\end{proof}

\begin{Definition}
Let $\P$ be a convex polyomino. A vertex $a\in V(\P)$ is called an interior vertex of $\P$, if $a$ is a vertex of four distinct cells of $\P$. We denote by $\i(\P)$ the set of all interior vertices of $\P$. The set $\partial\P=V(\P)\setminus \i(\P)$ is called the boundary of $\P$. We say that the vertex $a\in \partial\P$ is an inside (outside) corner of $\P$ if it belongs to exactly three (one) different cells of $\P$. (Figure~\ref{fig11})

\begin{figure}
  \centering
  \begin{tikzpicture}[domain=0:5]
        \filldraw[fill=black!5!white, draw=black] (0,0) rectangle (1,1) ;
        \filldraw[fill=black!5!white, draw=black] (1,0) rectangle (2,1) ;
        \filldraw[fill=black!5!white, draw=black] (2,0) rectangle (3,1) ;
        \filldraw[fill=black!5!white, draw=black] (3,0) rectangle (4,1) ;
        \filldraw[fill=black!5!white, draw=black] (4,0) rectangle (5,1) ;
        \filldraw[fill=black!5!white, draw=black] (0,1) rectangle (1,2) ;
        \filldraw[fill=black!5!white, draw=black] (1,1) rectangle (2,2) ;
        \filldraw[fill=black!5!white, draw=black] (2,1) rectangle (3,2) ;
        \filldraw[fill=black!5!white, draw=black] (3,1) rectangle (4,2) ;
        \filldraw[fill=black!5!white, draw=black] (1,2) rectangle (2,3) ;
        \filldraw[fill=black!5!white, draw=black] (2,2) rectangle (3,3) ;
        \filldraw[fill=black!5!white, draw=black] (1,3) rectangle (2,4) ;
        \filldraw[fill=black!5!white, draw=black] (2,3) rectangle (3,4) ;
        \filldraw[fill=black!5!white, draw=black] (2,4) rectangle (3,5) ;
        \filldraw (1,2) circle (2pt);
        \filldraw (2,4) circle (2pt);
        \filldraw (3,2) circle (2pt);
        \filldraw (4,1) circle (2pt);
  \end{tikzpicture}
  \caption{Inside corners}\label{fig11}
\end{figure}
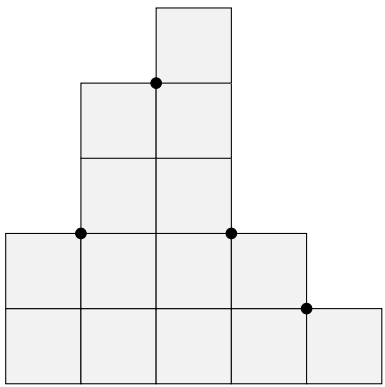
\end{Definition}

We may reformulate Corollary~\ref{cor1} as follows.

\begin{Corollary}\label{cor 2}
Let $\P$ be a convex stack polyomino and $[(1,1), (m,n)]$ the smallest interval which contains $V(\P)$. Then $\KK[\P]$ is Gorenstein if and only if $m=n$ and for each inside corner of $\P$ , by cutting all the cells of $\P$ which lie below the horizontal edge interval containing the corner, the minimal rectangle which contains the remaining polyomino is a square.
\end{Corollary}
\begin{proof}
If $\mathbb{K}[\P]$ is Gorenstein, then let $(x_r,y_t)$ be an inside corner of $\P$. We denote by $T$ the set $\{x\in X\mid \max\{j\in [n]\mid y_j\in N_Y(x)\}\leq t\}$. Then \[\max\{j\in[n]\mid y_j\in N_Y(T)\}=t<n.\] For $x\in X\setminus T$ we obtain $\max\{j\in [n]\mid y_j\in N_Y(x)\}>t$. By Corollary~\ref{cor1}, we have $|N_Y(T)|=|T|+1$. Thus, $|T|=t-1$. In other words, $n-t+1=n-|T|$ and the minimal rectangle we are interested in is a square.

Conversely, we suppose that $T\subset X$ is a set with the properties that $Y\setminus N_Y(T)\neq \emptyset$ and for every $x\in X\setminus T$, $$\max\{j\in [n]\mid y_j\in N_Y(x)\}>\max\{j\in [n]\mid y_j\in N_Y(T)\}.$$ Let $r=\max\{j\in [n]\mid y_j\in N_Y(T)\}<n$.

Since $\P$ is a column convex polyomino, $y_r$ is the $y$-coordinate of an inside corner. Then by hypothesis, $|X\setminus T|=n-r+1$. Hence, $n-|T|=n-r+1$ and $|T|+1=r=|N_Y(T)|$. By Corollary~\ref{cor1}, $\mathbb{K}[\P]$ is Gorenstein.\end{proof}

\begin{Examples}
By Corollary~\ref{cor 2}, the first polyomino of Figure~\ref{fig12} is Gorenstein, while the second is not.
\begin{figure}
  \centering
  \begin{tikzpicture}[domain=0:10]
        \filldraw[fill=black!5!white, draw=black] (0,0) rectangle (1,1);
        \filldraw[fill=black!5!white, draw=black] (1,0) rectangle (2,1) ;
        \filldraw[fill=black!5!white, draw=black] (2,0) rectangle (3,1) ;
        \filldraw[fill=black!5!white, draw=black] (3,0) rectangle (4,1) ;
        \filldraw[fill=black!5!white, draw=black] (0,1) rectangle (1,2) ;
        \filldraw[fill=black!5!white, draw=black] (1,1) rectangle (2,2) ;
        \filldraw[fill=black!5!white, draw=black] (2,1) rectangle (3,2) ;
        \filldraw[fill=black!5!white, draw=black] (3,1) rectangle (4,2) ;
        \filldraw[fill=black!5!white, draw=black] (1,2) rectangle (2,3) ;
        \filldraw[fill=black!5!white, draw=black] (2,2) rectangle (3,3) ;
        \filldraw[fill=black!5!white, draw=black] (2,3) rectangle (3,4) ;

        \filldraw[fill=black!5!white, draw=black] (6,0) rectangle (7,1);
        \filldraw[fill=black!5!white, draw=black] (7,0) rectangle (8,1) ;
        \filldraw[fill=black!5!white, draw=black] (8,0) rectangle (9,1) ;
        \filldraw[fill=black!5!white, draw=black] (9,0) rectangle (10,1) ;
        \filldraw[fill=black!5!white, draw=black] (7,1) rectangle (8,2) ;
        \filldraw[fill=black!5!white, draw=black] (8,1) rectangle (9,2) ;
        \filldraw[fill=black!5!white, draw=black] (7,2) rectangle (8,3) ;
        \filldraw[fill=black!5!white, draw=black] (8,2) rectangle (9,3) ;
        \filldraw[fill=black!5!white, draw=black] (8,3) rectangle (9,4) ;
  \end{tikzpicture}
  \caption{}\label{fig12}
\end{figure}
\end{Examples}

\section{The regularity of $\mathbb{K}[\P]$}

Let $\P$ be a stack polyomino on $[m]\times[n]$. Recall that the coordinate ring of $\P$ is a finitely generated module over the polynomial ring $S=\mathbb{K}[x_{ij}\mid (i,j)\in V(\P)]$. The Castelnuovo-Mumford regularity of $\mathbb{K}[\P]$, denoted $\reg(\mathbb{K}[\P])$, is defined to be the largest integer $r$ such that, for every $i$, the $i^{th}$ syzygy of $\mathbb{K}[\P]$ is generated in degree at most $r+i$.

We consider $H_{\mathbb{K}[\P]}(t)$ to be the Hilbert series of $\mathbb{K}[\P]$. Then $$H_{\mathbb{K}[\P]}(t)=\frac{Q(t)}{(1-t)^d}$$ where $Q(t)\in \mathbb{Z}[t]$ and where $d$ is the Krull dimension of $\mathbb{K}[\P]$. According to \cite[Theorem 2.2]{Q}, $d=\dim(\mathbb{K}[\P])=m+n-1$.

Since $\mathbb{K}[\P]$ is a Cohen-Macaulay ring, we have $$\reg(\mathbb{K}[\P])=\deg(Q(t))=\dim(\mathbb{K}[\P])+a(\mathbb{K}[\P]).$$ For the proof, we refer, for example, to \cite[Corollary B.4.1]{Va}.  

The $a$-invariant $a(\mathbb{K}[\P])$ of $\mathbb{K}[\P]$ is defined to be the degree of the Hilbert series of $\mathbb{K}[\P]$ that is, $a(\KK[\P])=\deg(Q(t))-d$.

Let $G_{\P}$ be the bipartite graph attached to $\P$. In this section, we consider $G_{\P}$ as a digraph with all its arrows leaving the vertex set $Y$. Following \cite{VV}, we introduce the following notion.
\begin{Definition}\label{directed cut}
If $T\subset X\cup Y$, then \[\delta^+(T)=\{e=(z,w)\in E(G_{\P})\mid z\in T\text{ and }w\notin T\}\] is the set of edges leaving the vertex set $T$ and \[\delta^-(T)=\{e=(z,w)\in E(G_{\P})\mid z\notin T\text{ and }w\in T\}\] is the set of edges entering the vertex set $T$. 

The set $\delta^+(T)$ is called a directed cut of the digraph $G_{\P}$ if $\emptyset\neq T\subsetneq X\cup Y$ and $\delta^-(T)=\emptyset$.
\end{Definition}
\begin{Example}
In the digraph of Figure~\ref{fig13}, let $T_1=\{x_3, y_2, y_3\}$ and $T_2=\{x_3, y_1, y_2\}$. Then we notice that
\[\emptyset\neq\delta^+(T_1)=\{(y_2, x_1), (y_2,x_2), (y_3, x_1), (y_3, x_2)\}\text{ and }\delta^-(T_1)=\{(y_1, x_3)\}\neq\emptyset,\] while
\[\emptyset\neq\delta^+(T_2)=\{(y_1, x_1), (y_1,x_2), (y_2, x_1), (y_2, x_2)\}\text{ and }\delta^-(T_2)=\emptyset.\]
Thus, $\delta^+(T_2)$ is a directed cut, while $\delta^+(T_1)$ is not.

\begin{figure}
  \centering
  \begin{tikzpicture}[domain=0:10]
        \filldraw[fill=black!5!white, draw=black] (0,0) rectangle (1,1);
        \filldraw[fill=black!5!white, draw=black] (0,1) rectangle (1,2) ;
        \filldraw[fill=black!5!white, draw=black] (0,2) rectangle (1,3) ;
        \filldraw[fill=black!5!white, draw=black] (1,0) rectangle (2,1) ;
        \filldraw (4,3) circle (0.25pt) node[above] {$x_1$};
        \filldraw (5,3) circle (0.25pt) node[above] {$x_2$};
        \filldraw (6,3) circle (0.25pt) node[above] {$x_3$};
        \filldraw (4,1) circle (0.25pt) node[below] {$y_1$};
        \filldraw (5,1) circle (0.25pt) node[below] {$y_2$};
        \filldraw (6,1) circle (0.25pt) node[below] {$y_3$};
        \filldraw (7,1) circle (0.25pt) node[below] {$y_4$};
        \draw[->,black] (4,1) -- (4,3);
        \draw[->,black] (4,1) -- (5,3);
        \draw[->,black] (4,1) -- (6,3);
        \draw[->,black] (5,1) -- (4,3);
        \draw[->,black] (5,1) -- (5,3);
        \draw[->,black] (5,1) -- (6,3);
        \draw[->,black] (6,1) -- (4,3);
        \draw[->,black] (6,1) -- (5,3);
        \draw[->,black] (7,1) -- (4,3);
        \draw[->,black] (7,1) -- (5,3);
  \end{tikzpicture}
  \caption{A stack polyomino and its associated digraph}\label{fig13}
\end{figure}
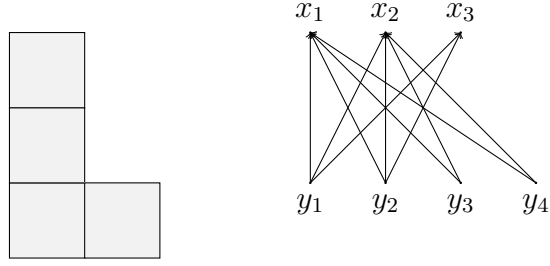
\end{Example}
\begin{Remarks}
Since $\KK[\P]\cong\KK[G_{\P}]$, we consider \[\delta^+(T)=\{(x,y)\in V(\P)\mid x\notin T\text{ and }y\in T\}\] and \[\delta^-(T)=\{(x,y)\in V(\P)\mid x\in T\text{ and } y\notin T\}\] for all $T\subset X\cup Y$.
If $T=X$, then $\delta^+(T)=\emptyset$. If $T\subsetneq Y$, then $\delta^-(T)=\emptyset$ and $\delta^+(T)$ is a directed cut of $G_{\P}$.
\end{Remarks}
\begin{Lemma}\label{dir_cut}
Let $\emptyset\neq T\subsetneq X\cup Y$. Then $\delta^+(T)$ is a directed cut of the digraph $G_{\P}$ if and only if $T=T^x\cup T^y$ with $T^x\subset X$, $T^y\subset Y$ and $N_Y(T^x)\subset T^y$.
\end{Lemma}
\begin{proof}
Let $T\neq\emptyset$ be a subset in $X\cup Y$. Then $T=T^x\cup T^y$ with $T^x\subset X$ and $T^y\subset Y$. By Definition~\ref{directed cut}, \[\delta^+(T)=\{(x,y)\in V(\P)\mid x\notin T^x\text{ and }y\in T^y\}\] is a directed cut of $G_{\P}$ if and only if \[\delta^-(T)=\{(x,y)\mid x\in T^x\text{ and }y\notin T^y\}=\emptyset.\] Suppose that $N_Y(T^x)\nsubseteq T^y$. Then there exist $x\in T^x$ and $y\in Y\setminus T^y$ such that $(x,y)\in V(\P)$. In other words, $(x,y)\in \delta^-(T)=\emptyset$, which is impossible.

Conversely, suppose that $N_Y(T^x)\subset T^y$ and $\delta^-(T)\neq\emptyset$. Then we find $x\in T^x$ and $y\in Y\setminus T^y$ such that $(x,y)\in V(\P)$. This is equivalent to saying that $y\in N_Y(x)\setminus T^y\subset N_Y(T^x)\setminus T^y$, which is false.
\end{proof}
In \cite{VV}, Valencia and Villarreal show that for any connected bipartite graph $G$, the $a$-invariant, $a(\KK[G])$ can be interpreted in combinatorial terms as follows.
\begin{Proposition}\label{a-inv}\cite[Proposition 4.2]{VV}
Let $G$ be a connected bipartite graph with $V(G)=X\cup Y$. If $G$ is a digraph with all its arrows leaving the vertex set $Y$, then
\[-a(\KK[G])=\text{the maximum number of disjoint directed cuts}.\]
\end{Proposition}
\begin{Example}
In the digraph of Figure~\ref{fig13}, $-a(\KK[G_{\P}])=4$ and a set of disjoint directed cuts is $\{\delta^+(\{y_1\}), \delta^+(\{y_2\}),\delta^+(\{y_3\}),\delta^+(\{y_4\})\}$.
\end{Example}
\begin{Lemma}\label{disj-cut}
Let $T_1=T_1^x\cup T_1^y$ and $T_2=T_2^x\cup T_2^y$ be two subsets in $X\cup Y$ such that $T_1^x, T_2^x\subset X$, $T_1^y, T_2^y\subset Y$ and $\delta^+(T_1)$, $\delta^+(T_2)$ are directed cuts. 

Then $\delta^+(T_1)\cap\delta^+(T_2)=\emptyset$ if and only if $T_1^x\cup T_2^x=X$ or $T_1^y\cap T_2^y=\emptyset$.
\end{Lemma}
\begin{proof}
Since \[\delta^+(T_1)=\{(x,y)\in V(\P)\mid x\notin T_1^x\text{ and }y\in T_1^y\}\] and \[\delta^+(T_2)=\{(x,y)\in V(\P)\mid x\notin T_2^x\text{ and }y\in T_2^y\},\] \[\delta^+(T_1)\cap \delta^+(T_2)=\{(x,y)\in V(\P)\mid x\notin T_1^x\cup T_2^x\text{ and }y\in T_1^y\cap T_2^y\}.\] Thus, $\delta^+(T_1)\cap \delta^+(T_2)=\emptyset$ if and only if $X\setminus(T_1^x\cup T_2^x)=\emptyset$ or $T_1^y\cap T_2^y=\emptyset$.
\end{proof}
\begin{Remark}\label{imp}
Let $\P$ be a stack polyomino on $[m]\times[n]$. Then
\[\delta^+(\{y_1\})=\{(x, y_1)\in V(\P)\mid x\in N_X(y_1)\}=N_X(y_1)\times \{y_1\},\]
\[\delta^+(\{y_2\})=\{(x, y_2)\in V(\P)\mid x\in N_X(y_2)\}=N_X(y_2)\times \{y_2\},\]
\[\vdots\]
\[\delta^+(\{y_n\})=\{(x, y_n)\in V(\P)\mid x\in N_X(y_n)\}=N_X(y_n)\times \{y_n\}\]
are disjoint directed cuts and also,
\[\delta^+(\{x_2,x_3,\ldots, x_m,y_1,y_2,\ldots,y_n\})=\{(x_1, y)\in V(\P)\mid y\in N_Y(x_1)\}=\]\[ =\{x_1\}\times N_Y(x_1),\]
\[\delta^+(\{x_1,x_3,\ldots,x_m,y_1,y_2, \ldots,y_n\})=\{(x_2, y)\in V(\P)\mid y\in N_Y(x_2)\}=\]\[=\{x_2\}\times N_Y(x_2),\]
\[\vdots\]
\[\delta^+(\{x_1,x_2,\ldots,x_{m-1},y_1,\ldots,y_n\})=\{(x_m, y)\in V(\P)\mid y\in N_Y(x_m)=\]\[=\{x_m\}\times N_Y(x_m)\]
are directed disjoint cuts.

By Proposition~\ref{a-inv}, we notice that $-a(\KK[\P])\geq\max\{m,n\}$.
\end{Remark}
\begin{Lemma}\label{a-invpoly}
If $\P$ is a stack polyomino on $[m]\times[n]$, then \[-a(\KK[\P])\leq\max\{m,n\}.\]
\end{Lemma}
\begin{proof}
Let $\P$ be a stack polyomino on $[m]\times[n]$ and $G_{\P}$ its associated bipartite graph. Since $\KK[\P]\cong\KK[G_{\P}]$, we have \[-a(\KK[\P])=\text{the maximum number of disjoint directed cuts},\] by Proposition~\ref{a-inv}.

We suppose that there are $\delta^+(T_1),\delta^+(T_2),\ldots, \delta^+(T_p)$ disjoint directed cuts with $T_1\subset X \cup Y, \ldots, T_p\subset X\cup Y$ and $p>\max\{m,n\}$. Moreover, for every $i$, we consider $T_i=T_i^x\cup T_i^y$ with $T_i^x\subset X$ and $T_i^y\subset Y$. If $T_i^x=X$ or $T_i^y=\emptyset$, then $\delta^+(T_i)=\emptyset$. Therefore, we may suppose that $T_i^x\neq X$ and $T_i^y\neq\emptyset$ for every $i\in [p]$.

If for every $i,j\in [p]$ with $i\neq j$ we have $T_i^y\cap T_j^y=\emptyset$, then \[\max\{m,n\}\geq n=|Y|\geq|\cup_{i=1}^pT_i^y|=\sum_{i=1}^{p}|T_i^y|\geq p,\] which is false.

If there are $i_1$ and $i_2\in[p]$ such that $T_{i_1}^y\cap T_{i_2}^y\neq\emptyset$, then $T_{i_1}^x\cup T_{i_2}^x=X$, by Lemma~\ref{disj-cut}. Since $\P$ is a stack polyomino, there is $x\in X=T_{i_1}^x\cup T_{i_2}^x$ such that $N_Y(x)=Y$. Without loss of generality, we suppose that $x\in T_{i_1}^x$. By Lemma~\ref{dir_cut}, this implies that $Y=N_Y(T_{i_1}^x)\subset T_{i_1}^y$. Thus, $T_{i_1}^y=Y$.

Let $j\in[p]\setminus\{i_1\}$. Then $T_j^y\cap T_{i_1}^y=T_j^y\neq\emptyset$. By Lemma~\ref{disj-cut}, we have $T_j^x\cup T_{i_1}^x=X$. Thus, we obtain that $T_i^x\neq\emptyset$ for every $i\in [p]$. Because $\P$ is a stack polyomino, $\{y_1, y_2\}\subset N_Y(T_i^x)\subset T_i^y$ for every $i\in[p]$. In other words, $T_i^y\cap T_j^y\neq\emptyset$ for every $i,j\in [p]$ with $i\neq j$. This implies that $T_i^x\cup T_j^x=X$ for every $i,j\in [p]$ with $i\neq j$. This is equivalent to saying that $(X\setminus T_i^x)\cap (X\setminus T_j^x)=\emptyset$ for every $i,j\in [p]$ with $i\neq j$. We note that $X\setminus T_i^x\neq\emptyset$ for every $i\in [p]$. Thus, \[\max\{m,n\}\geq m=|X|\geq |\cup_{i=1}^{p}C(T_i^x)|=\sum_{i=1}^{p}|C(T_i^x)|\geq p,\] which is false.
\end{proof}

Based on Proposition~\ref{a-inv}, Lemma~\ref{a-invpoly} and Remark~\ref{imp}, we obtain the following result.
\begin{Theorem}\label{p}
If $\P$ is a stack polyomino on $[m]\times[n]$, then \[-a(\mathbb{\KK[\P]})=\max\{m,n\}.\]
\end{Theorem}
\begin{Corollary}\label{regularity}
If $\P$ is a stack polyomino on $[m]\times[n]$, then \[\reg(\KK[\P])=\min\{m,n\}-1.\]
\end{Corollary}
\begin{Example}
Let $\P$ be the stack polyomino of Figure~\ref{fig14}. The Catelnuovo-Mumford regularity of $\KK[\P]$ is $\reg(\KK[\P])=3$.
\begin{figure}
  \centering
  \begin{tikzpicture}[domain=0:10]
        \filldraw[fill=black!5!white, draw=black] (0,0) rectangle (1,1);
        \filldraw[fill=black!5!white, draw=black] (1,0) rectangle (2,1);
        \filldraw[fill=black!5!white, draw=black] (2,0) rectangle (3,1);
        \filldraw[fill=black!5!white, draw=black] (3,0) rectangle (4,1);
        \filldraw[fill=black!5!white, draw=black] (4,0) rectangle (5,1);
        \filldraw[fill=black!5!white, draw=black] (1,1) rectangle (2,2);
        \filldraw[fill=black!5!white, draw=black] (2,1) rectangle (3,2);
        \filldraw[fill=black!5!white, draw=black] (3,1) rectangle (4,2);
        \filldraw[fill=black!5!white, draw=black] (2,2) rectangle (3,3);
        \filldraw[fill=black!5!white, draw=black] (3,2) rectangle (4,3);
  \end{tikzpicture}
  \caption{}\label{fig14}
\end{figure}
\end{Example}

\section{The multiplicity of $\mathbb{K}[\P]$}

Let $\P$ be a stack polyomino on $[m]\times[n]$. The Hilbert series $H_{\mathbb{K}[\P]}(t)$ of $\mathbb{K}[\P]$ is given by $$H_{\mathbb{K}[\P]}(t)=\frac{Q(t)}{(1-t)^d},$$ where $Q(t)\in \mathbb{Z}[t]$ and $d$ is the Krull dimension of $\mathbb{K}[\P]$. According to \cite[Theorem 2.2]{Q}, $d=\dim(\mathbb{K}[\P])=m+n-1$. The multiplicity of $\mathbb{K}[\P]$, denoted $e(\mathbb{K}[\P])$, is given by $Q(1)$.

For every $i\in[m]$, we define the height of $i$ as \[\height(i)=\max\{j\in [n]\mid (i,j)\in V(\P)\}.\]
Following the proof of \cite[Theorem]{OTH}, we give a total order on the variables $x_{ij}$, with $(i,j)\in V(\P)$, as follows:
\begin{equation}\label{order}
x_{ij}>x_{kl}\text{ if and only if }
\end{equation}
\[(\height(i)>\height(k))\text{ or }(\height(i)=\height(k)\text{ and }i>k)\text{ or }(i=k \text{ and }j>l).\]
Let $<$ be the reverse lexicographical order induced by this order of variables. As we have already seen in the previous sections, the ideal $I_{\P}$ can be viewed as the toric ideal of the edge ring $\KK[G_{\P}]$, where $G_{\P}$ is the bipartite graph associated to $\P$. As it follows from the proof of \cite[Theorem]{OTH}, the reduced Gr\"obner basis of $I_{\P}$ with respect to $<$ consists of all $2$-inner minors of $\P$. In what follows, whenever we consider the Gr\"obner basis of $I_{\P}$, we assume that the variables $x_{ij}$, with $(i,j)\in V(\P)$ are totally order as in (\ref{order}).

We notice that $\ini_{<}(I_{\P})$ is a squarefree monomial ideal. Thus, we may view $\ini_<(I_{\P})$ as the Stanley-Reisner ideal of a simplicial complex on the vertex set $V(\P)$. Let $\Delta_{\P}$ denote this simplicial complex. It is known that $\Delta_{\P}$ is a pure shellable simplicial complex by \cite[Theorem 9.6.1]{V} and \cite[Theorem 9.5.10]{LRS}.

Let $f=(f_0,f_1,\cdots,f_{d-1})$ be the $f$-vector of $\Delta_{\P}$, where $d=m+n-1$.

We have \[H_{\mathbb{K}[\P]}(t)=H_{S/\ini_{<}(I_{\P})}(t)=H_{\KK[\Delta_{\P}]}.\]
By \cite[Lemma 5.1.8]{BH}, \[e(\KK[\P])= f_{d-1},\] that is, $e(\KK[\P])$ is equal to the number of facets of $\Delta_{\P}$. Thus, $e(\KK[\P])=|\mathfrak{F}(\Delta_{\P})|$, where $\mathfrak{F}(\Delta_{\P})$ denotes the set of the facets of $\Delta_{\P}$.
\begin{Example}
Let $\P$ be the polyomino of Figure~\ref{ex1}. We order the variables as follows $x_{23}>x_{22}>x_{21}>x_{13}>x_{12}>x_{11}>x_{32}>x_{31}$. Then with respect to the reverse lexicographical order induced by this order of variables, we have
\[\ini_{<}(I_{\P})=(x_{11}x_{32}, x_{21}x_{32}, x_{21}x_{12}, x_{21}x_{13}, x_{22}x_{13})\text{ and }\]
\[\Delta_{\P}=\langle F_1=\{(1,1), (2,1), (2,2), (2,3), (3,1)\};\]
\[F_2=\{(1,1), (1,2), (2,2), (2,3), (3,1)\}; F_3=\{(1,1), (1,2), (1,3), (2,3), (3,1)\};\]
\[F_4=\{(1,2), (2,2), (2,3), (3,1), (3,2)\}; F_5=\{(1,2), (1,3), (2,3), (3,1), (3,2)\}\rangle.\]
\begin{figure}
  \centering
  \begin{tikzpicture}[domain=0:10]
        \filldraw[fill=black!5!white, draw=black] (0,0) rectangle (1,1);
        \filldraw[fill=black!5!white, draw=black] (1,0) rectangle (2,1) ;
        \filldraw[fill=black!5!white, draw=black] (0,1) rectangle (1,2) ;
  \end{tikzpicture}
  \caption{}\label{ex1}
\end{figure}
\end{Example}

\begin{Definition}
Let $\Delta$ be a simplicial complex on the vertex set $V$ and $v\in V$. The link of $v$ in $\Delta$ is the simplicial complex
\[\lk(v)=\{F\in \Delta\mid v\notin F\text{ and }F\cup \{v\}\in \Delta\}\]
and the deletion of $v$ is the simplicial complex
\[\del(v)=\{F\in \Delta\mid v\notin F\}.\]
\end{Definition}

Let $x_{ij}$ be the smallest variable in $S$ with respect to $<$ and fix $v=(i,\height(i))\in V(\P)$. If $i=1$, then we denote by $\P_1$ the polyomino obtained from $\P$ by deleting the cell which contains the vertex $v$. Otherwise, $\P_1$ is given by deleting the cell which contains the vertex $(m,\height(m))$; see Figure~\ref{figg}.
\begin{figure}
  \centering
  \begin{tikzpicture}[domain=0:10]
        \filldraw[fill=black!5!white, draw=black] (1,1) rectangle (2,2);
        \filldraw[fill=black!5!white, draw=black] (1,2) rectangle (2,3);
        \filldraw[fill=black!5!white, draw=black] (1,3) rectangle (2,4);
        \filldraw[fill=black!5!white, draw=black] (1,6) rectangle (2,7);
        \filldraw[fill=black!5!white, draw=black] (2,1) rectangle (3,2);
        \filldraw[fill=black!5!white, draw=black] (2,2) rectangle (3,3);
        \filldraw[fill=black!5!white, draw=black] (2,6) rectangle (3,7);
        \filldraw[fill=black!5!white, draw=black] (2,7) rectangle (3,8);
        \filldraw[fill=black!5!white, draw=black] (3,1) rectangle (4,2);
        \filldraw[fill=black!5!white, draw=black] (3,2) rectangle (4,3);
        \filldraw[fill=black!5!white, draw=black] (3,6) rectangle (4,7);
        \filldraw[fill=black!5!white, draw=black] (3,7) rectangle (4,8);
        \filldraw[fill=black!5!white, draw=black] (3,8) rectangle (4,9);
        \filldraw[fill=black!5!white, draw=black] (4,1) rectangle (5,2);
        \filldraw[fill=black!5!white, draw=black] (4,2) rectangle (5,3);
        \filldraw[fill=black!5!white, draw=black] (4,6) rectangle (5,7);
        \filldraw[fill=black!5!white, draw=black] (4,7) rectangle (5,8);
        \filldraw[fill=black!5!white, draw=black] (6,1) rectangle (7,2);
        \filldraw[fill=black!5!white, draw=black] (6,2) rectangle (7,3);
        \filldraw[fill=black!5!white, draw=black] (6,3) rectangle (7,4);
        \filldraw[fill=black!5!white, draw=black] (6,6) rectangle (7,7);
        \filldraw[fill=black!5!white, draw=black] (6,7) rectangle (7,8);
        \filldraw[fill=black!5!white, draw=black] (7,1) rectangle (8,2);
        \filldraw[fill=black!5!white, draw=black] (7,2) rectangle (8,3);
        \filldraw[fill=black!5!white, draw=black] (7,6) rectangle (8,7);
        \filldraw[fill=black!5!white, draw=black] (7,7) rectangle (8,8);
        \filldraw[fill=black!5!white, draw=black] (7,8) rectangle (8,9);
        \filldraw[fill=black!5!white, draw=black] (8,1) rectangle (9,2);
        \filldraw[fill=black!5!white, draw=black] (8,2) rectangle (9,3);
        \filldraw[fill=black!5!white, draw=black] (8,6) rectangle (9,7);
        \filldraw[fill=black!5!white, draw=black] (8,7) rectangle (9,8);
        \filldraw[fill=black!5!white, draw=black] (9,1) rectangle (10,2);
        \filldraw (3, 4.5) circle (0 pt) node[above] {$\P$};
        \filldraw (3, 9.5) circle (0 pt) node[above] {$\P$};
        \filldraw (7.5, 4.5) circle (0 pt) node[above] {$\P_1$};
        \filldraw (7.5, 9.5) circle (0 pt) node[above] {$\P_1$};
        \filldraw (1, 7) circle (1 pt) node[above] {$v$};
        \filldraw (5, 3) circle (1 pt) node[above] {$v$};
  \end{tikzpicture}
  \caption{}\label{figg}
\end{figure}
\begin{Remark}\label{little}
Since $x_{i1}$ is the smallest variable with respect to $<$, we have $(i,1)\in F$ for every $F\in \Delta_{\P}$. Indeed $x_{i1}$ is regular on $S/\ini_{<}(I_{\P})$, thus it does not belong to any of the minimal primes of $\ini_<(I_{\P})$ which implies that $x_{i1}$ belongs to all the facets of $\Delta_{\P}$.
\end{Remark}

\begin{Lemma}\label{delta1}
We have $|\mathfrak{F}(\Delta_{\P_1})|=|\mathfrak{F}(\del(v))|$.
\end{Lemma}
\begin{proof}
To begin with, let us consider $\height(i)\geq 3$. If $F\in\mathfrak{F}(\del(v))$, then we obtain $F'\in\mathfrak{F}(\Delta_{\P_1})$ by Algorithm~\ref{alg1}.
\alglanguage{pascal}
\begin{algorithm}
\caption{}\label{alg1}
\begin{algorithmic}[1]
\Begin
\State $F':=F$;
\State $h:=\height(i)$;
\If{i\neq 1}
\Begin
    \For{k=1}{h}
    \Begin
        \If{(m,k)\in F}
            \State $F':=F'\setminus\{(m,k)\}$;
        \If{(i,k)\in F}
            \State $F':=(F'\setminus \{(i,k)\})\cup\{(m,k)\}$;
    \End
    \For{j=i+1}{m-1}
        \For{k=1}{h}
            \If{(j,k)\in F}
                \State $F':=(F'\setminus\{(j,k)\})\cup\{(j-1,k)\}$;
    \For{k=1}{h}
        \If{(m,k)\in F}
            \State $F':=F'\cup\{(m-1,k)\}$;
\End
\End
\end{algorithmic}
\end{algorithm}

Indeed, if $F\in \mathfrak{F}(\del(v))$ and $i\neq 1$, then $v\notin F$ and $|F|=m+n-1$. By applying the first "For" loop from Algorithm~\ref{alg1}, we already obtain $(m,\height(m))\notin F'$ and so $F'\subset V(\P_1)$. Next we apply the following "For" loops. Since $F'$ is obtained from $F$ by a circular permutation of the vertices of $F$ which have the $x$-coordinate greater than $i$, we get $|F|=|F'|=m+n-1=\dim \Delta_{\P_1}+1$.

For example, in the polyomino $\P$ of Figure~\ref{figa}, $v=(3,3)$. We illustrate all the "For" loops of Algortihm~\ref{alg1} in Figure~\ref{figa} for the facet
\[F=\{(1,3),(1,4),(2,4),(3,1),(3,2),(4,2),(4,3),(5,3),(6,3)\}\in \del(v).\] Following the algorithm, \[F'=\{(1,3),(1,4),(2,4),(3,2),(3,3),(4,3),(5,3),(6,1),(6,2)\}.\]
\begin{figure}
  \centering
  \begin{tikzpicture}[domain=1:13]
        \filldraw[fill=black!5!white, draw=black] (1,1) rectangle (2,2);
        \filldraw[fill=black!5!white, draw=black] (1,2) rectangle (2,3);
        \filldraw[fill=black!5!white, draw=black] (1,3) rectangle (2,4);
        \filldraw[fill=black!5!white, draw=black] (2,1) rectangle (3,2);
        \filldraw[fill=black!5!white, draw=black] (2,2) rectangle (3,3);
        \filldraw[fill=black!5!white, draw=black] (3,1) rectangle (4,2);
        \filldraw[fill=black!5!white, draw=black] (3,2) rectangle (4,3);
        \filldraw[fill=black!5!white, draw=black] (4,1) rectangle (5,2);
        \filldraw[fill=black!5!white, draw=black] (4,2) rectangle (5,3);
        \filldraw[fill=black!5!white, draw=black] (5,1) rectangle (6,2);
        %\filldraw[fill=black!0!white, draw=black] (5,2) rectangle (6,3);
        \draw[-,black, dashed] (5,3) -- (6,3) -- (6,2);
        \filldraw[fill=black!5!white, draw=black] (8,1) rectangle (9,2);
        \filldraw[fill=black!5!white, draw=black] (8,2) rectangle (9,3);
        \filldraw[fill=black!5!white, draw=black] (8,3) rectangle (9,4);
        \filldraw[fill=black!5!white, draw=black] (9,1) rectangle (10,2);
        \filldraw[fill=black!5!white, draw=black] (9,2) rectangle (10,3);
        \filldraw[fill=black!5!white, draw=black] (10,1) rectangle (11,2);
        \filldraw[fill=black!5!white, draw=black] (10,2) rectangle (11,3);
        \filldraw[fill=black!5!white, draw=black] (11,1) rectangle (12,2);
        \filldraw[fill=black!5!white, draw=black] (11,2) rectangle (12,3);
        \filldraw[fill=black!5!white, draw=black] (12,1) rectangle (13,2);
        %\filldraw[fill=black!0!white, draw=black] (12,2) rectangle (13,3);
        \draw[-,black, dashed] (12,3) -- (13,3)--(13,2);
        \filldraw[fill=black!5!white, draw=black] (1,6) rectangle (2,7);
        \filldraw[fill=black!5!white, draw=black] (1,7) rectangle (2,8);
        \filldraw[fill=black!5!white, draw=black] (1,8) rectangle (2,9);
        \filldraw[fill=black!5!white, draw=black] (2,6) rectangle (3,7);
        \filldraw[fill=black!5!white, draw=black] (2,7) rectangle (3,8) circle (0 pt) node[above] {$v$};
        \filldraw[fill=black!5!white, draw=black] (3,6) rectangle (4,7);
        \filldraw[fill=black!5!white, draw=black] (3,7) rectangle (4,8);
        \filldraw[fill=black!5!white, draw=black] (4,6) rectangle (5,7);
        \filldraw[fill=black!5!white, draw=black] (4,7) rectangle (5,8);
        \filldraw[fill=black!5!white, draw=black] (5,6) rectangle (6,7);
        \filldraw[fill=black!5!white, draw=black] (5,7) rectangle (6,8);
        \filldraw[fill=black!5!white, draw=black] (8,6) rectangle (9,7);
        \filldraw[fill=black!5!white, draw=black] (8,7) rectangle (9,8);
        \filldraw[fill=black!5!white, draw=black] (8,8) rectangle (9,9);
        \filldraw[fill=black!5!white, draw=black] (9,6) rectangle (10,7);
        \filldraw[fill=black!5!white, draw=black] (9,7) rectangle (10,8);
        \filldraw[fill=black!5!white, draw=black] (10,6) rectangle (11,7);
        \filldraw[fill=black!5!white, draw=black] (10,7) rectangle (11,8);
        \filldraw[fill=black!5!white, draw=black] (11,6) rectangle (12,7);
        \filldraw[fill=black!5!white, draw=black] (11,7) rectangle (12,8);
        \filldraw[fill=black!5!white, draw=black] (12,6) rectangle (13,7);
        %\filldraw[fill=black!0!white, draw=black] (12,7) rectangle (13,8);
        \draw[-,black, dashed] (12,8) -- (13,8) -- (13,7);
        %\filldraw (0, 4) circle (1 pt) node[above] {$\P_1$};
        \filldraw (1, 8) circle (2 pt);
        \filldraw (1, 9) circle (2 pt);
        \filldraw (2, 9) circle (2 pt);
        \filldraw (3, 6) circle (2 pt);
        \filldraw (3, 7) circle (2 pt);
        \filldraw (4, 7) circle (2 pt);
        \filldraw (4, 8) circle (2 pt);
        \filldraw (5, 8) circle (2 pt);
        \filldraw (6, 8) circle (2 pt);
        \filldraw (8, 8) circle (2 pt);
        \filldraw (8, 9) circle (2 pt);
        \filldraw (9, 9) circle (2 pt);
        \filldraw (11, 7) circle (2 pt);
        \filldraw (11, 8) circle (2 pt);
        \filldraw (12, 8) circle (2 pt);
        \filldraw (13, 7) circle (2 pt);
        \filldraw (13, 6) circle (2 pt);
        \filldraw (1, 3) circle (2 pt);
        \filldraw (1, 4) circle (2 pt);
        \filldraw (2, 4) circle (2 pt);
        \filldraw (3, 2) circle (2 pt);
        \filldraw (3, 3) circle (2 pt);
        \filldraw (4, 3) circle (2 pt);
        \filldraw (6, 1) circle (2 pt);
        \filldraw (6, 2) circle (2 pt);
        \filldraw (8, 3) circle (2 pt);
        \filldraw (8, 4) circle (2 pt);
        \filldraw (9, 4) circle (2 pt);
        \filldraw (10, 2) circle (2 pt);
        \filldraw (10, 3) circle (2 pt);
        \filldraw (11, 3) circle (2 pt);
        \filldraw (12, 3) circle (2 pt);
        \filldraw (13, 1) circle (2 pt);
        \filldraw (13, 2) circle (2 pt);
        \filldraw (3, 9.5) circle (0 pt) node[above] {$\P$};
        \filldraw (10, 9.5) circle (0 pt) node[above] {The first "For" loop};
        \filldraw (3, 4.5) circle (0 pt) node[above] {The second "For" loop};
        \filldraw (10, 4.5) circle (0 pt) node[above] {The third "For" loop};
  \end{tikzpicture}
  \caption{}\label{figa}
\end{figure}

Now, we observe that even if the order of the variables for $\P_1$ is not induced by the order of the variables of $\P$, the generators of $\ini_<(I_{\P_1})$ are exactly those which appear in $\ini_<(I_{\P})$. Therefore, we may conclude that $F'\in\Delta_{\P_1}$ and so $F'\in \mathfrak{F}(\Delta_{\P_1})$.

In the case that $i=1$, we notice that $F=F'$ and $\mathfrak{F}(\Delta_{\P_1})=\mathfrak{F}(\del(v))$. In fact, if $F\in \mathfrak{F}(\Delta_{\P_1})$, then $v\notin F$ and $|F|=m+n-1$. Since $F\in\del(v)$ and $\dim\Delta_{\P}=m+n-2$, it follows that $F\in \mathfrak{F}(\del(v))$. For the other inclusion, let $F\in \mathfrak{F}(\del(v))$. Then $F\in\Delta_{\P_1}$. Since $\dim S/I_{\P_1}=m+n-1$, it follows that $F\in \mathfrak{F}(\Delta_{\P_1})$.

Therefore, we have shown that every facet $F$ of $\del(v)$ determines uniquely a facet $F'$ in $\Delta_{\P_1}$, if $\height(i)\geq 3$.

Conversely, let $F'$ be a facet in $\Delta_{\P_1}$. Following the steps of Algorithm~\ref{alg1} in reverse order, we obtain $F$ a facet in $\del(v)$. Algorithm~\ref{alg2} gives explicitly all the steps of getting $F$ from $F'$.
\alglanguage{pascal}
\begin{algorithm}
\caption{}\label{alg2}
\begin{algorithmic}[1]
\Begin
\State $F:=F'$;
\State $h:=\height(i)$;
\If{i\neq 1}
\Begin
    \For{k=1}{h}
        \If{(m-1,k)\in F'}
            \State $F:=F\setminus\{(m-1,k)\}$;
    \If{i\leq m-2}
    \For{j=m-2}{i}
        \For{k=1}{h}
            \If{(j,k)\in F'}
                \State $F:=(F\setminus\{(j,k)\})\cup\{(j+1,k)\}$;
    \For{k=1}{h}
    \Begin
        \If{(m,k)\in F'}
            \State $F:=(F\setminus\{(m,k)\})\cup\{(i,k)\}$;
        \If{(m-1,k)\in F'}
            \State $F:=F\cup\{(m,k)\}$;
    \End
\End
\End
\end{algorithmic}
\end{algorithm}

We thus get $|\mathfrak{F}(\Delta_{\P_1})|=|\mathfrak{F}(\del(v))|$ if $\height(i)\geq 3$. Moreover, we have equality between the sets $\mathfrak{F}(\Delta_{\P_1})$ and $\mathfrak{F}(\del(v))$ if and only if $i=1$ and $\height(i)\geq 3$.

In order to complete the proof, let us point out that the same two algorithms work for $\height(i)=2$. In fact, for $i>1$ (respectively $i=1$), $F$ is a facet in $\del(v)$ if and only if $F'$ is a facet in the cone $(m,1)*\Delta_{\P_1}$ (respectively $(1,1)*\Delta_{\P_1}$).

For example, if we consider the polyomino $\P$ of Figure~\ref{figb} and $v=(3,2)\in V(\P)$, then if we choose
\[F=\{(1,2),(1,3),(1,4),(2,4),(3,1),(4,1),(4,2),(5,2)\}\in \mathfrak{F}(\del(v)),\] we obtain
\[F'=\{(1,2),(1,3),(1,4),(2,4),(3,1),(3,2),(4,2),(5,1)\}\] a facet in $(5,1)*\Delta_{\P_1}$.
\begin{figure}
  \centering
  \begin{tikzpicture}[domain=1:13]
        \filldraw[fill=black!5!white, draw=black] (1,1) rectangle (2,2);
        \filldraw[fill=black!5!white, draw=black] (1,2) rectangle (2,3);
        \filldraw[fill=black!5!white, draw=black] (1,3) rectangle (2,4);
        \filldraw[fill=black!5!white, draw=black] (2,1) rectangle (3,2) circle (0 pt) node[above] {$v$};
        \filldraw[fill=black!5!white, draw=black] (3,1) rectangle (4,2);
        \filldraw[fill=black!5!white, draw=black] (4,1) rectangle (5,2);
        \filldraw[fill=black!5!white, draw=black] (6,1) rectangle (7,2);
        \filldraw[fill=black!5!white, draw=black] (6,2) rectangle (7,3);
        \filldraw[fill=black!5!white, draw=black] (6,3) rectangle (7,4);
        \filldraw[fill=black!5!white, draw=black] (7,1) rectangle (8,2);
        \filldraw[fill=black!5!white, draw=black] (8,1) rectangle (9,2);
        \filldraw[fill=black!0!white, draw=black, dashed] (9,1) rectangle (10,2);
        \filldraw (3, 4.5) circle (0 pt) node[above] {$\P$};
        \filldraw (8, 4.5) circle (0 pt) node[above] {$\P_1$};
        \filldraw (1, 2) circle (2 pt);
        \filldraw (1, 3) circle (2 pt);
        \filldraw (1, 4) circle (2 pt);
        \filldraw (2, 4) circle (2 pt);
        \filldraw (3, 1) circle (2 pt);
        \filldraw (4, 1) circle (2 pt);
        \filldraw (4, 2) circle (2 pt);
        \filldraw (5, 2) circle (2 pt);
        \filldraw (6, 2) circle (2 pt);
        \filldraw (6, 3) circle (2 pt);
        \filldraw (6, 4) circle (2 pt);
        \filldraw (7, 4) circle (2 pt);
        \filldraw (8, 1) circle (2 pt);
        \filldraw (8, 2) circle (2 pt);
        \filldraw (9, 2) circle (2 pt);
        \filldraw (10, 1) circle (2 pt);
        \draw[-,black] (9,2) -- (9,1);
  \end{tikzpicture}
  \caption{}\label{figb}
\end{figure}
\end{proof}
%\begin{Remark}
%Let us point out that $\del(v)$ represents the cone $(i,1)*\Delta_{\P_1}$, if $v=(i,2)$.
%\end{Remark}
Let $\P_2$ be the polyomino obtained from $\P$ by deleting all the cells of $\P$ which lie below the horizontal edge interval containing the vertex $v$.
\begin{Lemma}\label{delta2}
We have $|\mathfrak{F}(\Delta_{\P_2})|=|\mathfrak{F}(\lk(v))|$.
\end{Lemma}
\begin{proof}
  Let $F$ be a facet in $\lk(v)$. Then $F\cup\{v\}\in \mathfrak{F}(\Delta_{\P})$. In this proof, we set $j=\height(i)$.

  Suppose that $F\cup \{v\}=G_1\cup G_2$ where $G_1\in \Delta_{P_2}$ and $G_2=\{(a,j)\mid (a,j)\in V(\P)\setminus V(\P_2)\}\cup \{(i,j-1),\ldots, (i,1)\}$. In fact, since $v\in F\cup \{v\}$, all the vertices of $G_2$ must belong to $F\cup\{v\}$ and $x_{ij}x_{kl}\in \ini_{<}(I_{\P})$, for every $(k,l)\in V(\P)\setminus G_2$ with $l<j$.

  In order to prove that $G_1\in\mathfrak{F}(\Delta_{\P_2})$, it is enough to show that $|G_1|=\dim \Delta_{\P_2}+1$. We consider the polyomino $\P_2$ to be on $[m-t]\times[n-j+1]$, for some $t\geq 1$. It follows that $|F\cup \{v\}|=|G_1\cup G_2|=|G_1|+|G_2|\leq (m-t+n-j+1-1)+(t+j-1)=m+n-1$, which implies that $|G_1|=(m-t)+(n-j+1)-1=\dim \Delta_{\P_2}+1$. Therefore, $G_1\in\mathfrak{F}(\Delta_{\P_2})$.

  We consider $G$ to be a facet in $\Delta_{\P_2}$. By definition of $\P_2$, $v\notin G$ and $G\cup\{v\}\in \Delta_{\P}$. In other words, $G\in \lk(v)$ and there exists $F\in \mathfrak{F}(\lk(v))$ such that $G\subset F$. Thus, $F\cup\{v\}\in \mathfrak{F}(\Delta_{\P})$.  Moreover, $F\cup\{v\}=G\cup G_2$.
\end{proof}
%\begin{Remark}\label{formultipl}
%If $F\in\mathfrak{F}(\del(v))$, then $|F|=m+n-1$ and $F\in\mathfrak{F}(\Delta_{\P})$.

%If $F\in\mathfrak{F}(\lk(v))$, then $F\cup\{v\}\in\mathfrak{F}(\Delta_{\P})$ and $|F|<m+n-1$.
%\end{Remark}
We now prove the main result of this section.

\begin{Theorem}\label{mult}
Let $\P$ be a stack polyomino on $[m]\times [n]$ and $v=(i,j)\in V(\P)$ with the properties:
\begin{enumerate}
  \item $x_{i1}$ is the smallest variable in $S$ and
  \item $j=\height(i)$.
\end{enumerate}
We consider $\P_1$ and $\P_2$ to be the following polyominoes:
\begin{enumerate}
  \item $\P_1$ is the polyomino obtained from $\P$ by deleting the cell which contains the vertex $v$, if $i=1$. Otherwise, $\P_1$ is given by deleting the cell of $\P$ which contains the vertex $(m,\height(m))$.
  \item $\P_2$ is the polyomino obtained from $\P$ be deleting all the cells of $\P$ which lie below the horizontal edge interval containing the vertex $v$.
\end{enumerate}
Then
\[e(\KK[\P])=e(\KK[\P_1])+e(\KK[\P_2]).\]
\end{Theorem}
\begin{proof}
  In order to prove the equality, it is sufficient to show that \[|\mathfrak{F}(\Delta_{\P})|=|\mathfrak{F}(\Delta_{\P_1})|+|\mathfrak{F}(\Delta_{\P_2})|.\]
  We consider $F$ to be a facet in $\Delta_{\P}$. If $v\in F$, then $F\setminus\{v\}\in\mathfrak{F}(\lk(v))$. Otherwise, $v\notin F$, thus $F\in\mathfrak{F}(del(v))$. Therefore, we obtain $|\mathfrak{F}(\Delta_{\P})|=|\mathfrak{F}(\lk(v))|+|\mathfrak{F}(\del(v)|$. The proof is completed by applying Lemma~\ref{delta1} and Lemma~\ref{delta2}.
\end{proof}
\begin{Example}
Let $\P$ be the stack polyomino of Figure~\ref{ex3}. Then the multiplicity of $\KK[\P]$ is equal to $14$. The first step in the recursive formula, namely $e(\KK[\P])=e(\KK[\P_1])+e(\KK[\P_2])$, is shown in the figure. Next we apply the recursive procedure for each of the polyominoes $\P_1$ and $\P_2$.
\begin{figure}
  \centering
  \begin{tikzpicture}[domain=0:11]
        \filldraw[fill=black!5!white, draw=black] (0,7) rectangle (1,8);
        \filldraw[fill=black!5!white, draw=black] (0,8) rectangle (1,9);
        \filldraw[fill=black!5!white, draw=black] (1,7) rectangle (2,8);
        \filldraw[fill=black!5!white, draw=black] (1,8) rectangle (2,9);
        \filldraw[fill=black!5!white, draw=black] (1,9) rectangle (2,10);
        \filldraw[fill=black!5!white, draw=black] (2,7) rectangle (3,8);
        \filldraw[fill=black!5!white, draw=black] (5,7) rectangle (6,8);
        \filldraw[fill=black!5!white, draw=black] (5,8) rectangle (6,9);
        \filldraw[fill=black!5!white, draw=black] (6,7) rectangle (7,8);
        \filldraw[fill=black!5!white, draw=black] (6,8) rectangle (7,9);
        \filldraw[fill=black!5!white, draw=black] (6,9) rectangle (7,10);
        \filldraw[fill=black!5!white, draw=black] (9,8) rectangle (10,9);
        \filldraw[fill=black!5!white, draw=black] (10,8) rectangle (11,9);
        \filldraw[fill=black!5!white, draw=black] (10,9) rectangle (11,10);
        \filldraw (3, 8) circle (0 pt) node[above] {$v$};
        \filldraw (0, 10) circle (0 pt) node[above] {$\P$};
        \filldraw (5, 10) circle (0 pt) node[above] {$\P_1$};
        \filldraw (9, 10) circle (0 pt) node[above] {$\P_2$};
  \end{tikzpicture}
  \caption{$e(\KK[\P])=14$}\label{ex3}
\end{figure}
\end{Example}
\begin{Example}\label{ex}
Let $\P_{m,n}$ be the stack polyomino on $[m]\times[n]$ with $V(\P_{m,n})=[m]\times [n]$. The multiplicity of $\KK[\P_{m,n}]$ was computed in \cite{HT} and \[e(\KK[\P_{m,n}])={m+n-2\choose m-1}.\]
Now, we consider $k<n$ to be a positive integer and $\P_k$ to be the polyomino of Figure~\ref{ex4}. It consists of a rectangle of size $[m-1]\times [n]$ together with a column of cells of height equal to $k$. By Theorem~\ref{mult},
\[e(\KK[\P_k])=e(\KK[\P_{k-1}])+e(\KK[\P_{m-1,n-k}])=e(\KK[\P_{k-1}])+{{m+n-k-3}\choose{m-2}}.\]
Applying recursively this formula, we obtain
\[e(\KK[\P_k])={m+n-3 \choose m-2}+{m+n-4 \choose m-2}+\cdots+{m+n-k-3\choose m-2}=\]\[={m+n-2 \choose m-1}-{m+n-k-3\choose m-1}.\]
\begin{figure}
    \centering
      \begin{tikzpicture}[domain=0:12]
        \draw[-,black] (0,0) -- (0,7) -- (3,7) -- (3,4) -- (4,4) -- (4,0) -- (0,0);
        \draw[-,black, dashed] (3,0) -- (3,4);
        \filldraw (4,4) circle (0.75pt) node[right] {$(m,k)$};
        \draw[-,black] (0,-0.5) -- node[near start,below] {$m-1$} (3,-0.5);
        \draw[-,black] (-0.5,0) -- node[near start,left] {$n$} (-0.5,7);
    \end{tikzpicture}
    \caption{}\label{ex4}
\end{figure}
\end{Example}
\begin{Example}
Let $\P(m,n,k_1, k_2,\ldots,k_l)$ be the polyomino of Figure~\ref{ex5}. This is an example of one sided ladder with the last $l$ columns of heights $k_1,\ldots, k_l$. Hilbert series of one sided ladders have been considered in \cite{W}. By Theorem~\ref{mult},
\[e(\KK[\P(m,n,k_1, k_2,\ldots,k_l)])=\]\[e(\KK[\P(m-1,n-k_l+1,k_1-k_l+1, k_2-k_l+1,\ldots,k_{l-1}-k_l+1)])+\]\[+e(\KK[\P(m,n,k_1, k_2,\ldots,k_{l-1},k_l-1)])\]\[=e(\KK[\P(m-1,n-k_l+1,k_1-k_l+1, k_2-k_l+1,\ldots,k_{l-1}-k_l+1)])+\]
\[+e(\KK[\P(m-1,n-k_l+2,k_1-k_l+2, k_2-k_l+2,\ldots,k_{l-1}-k_l+2)])+\]\[+e(\KK[\P(m,n,k_1, k_2,\ldots,k_{l-1},k_l-2)])=\]
\[\cdots=e(\KK[\P(m-1,n-k_l+1,k_1-k_l+1, k_2-k_l+1,\ldots,k_{l-1}-k_l+1)])+\]\[+e(\KK[\P(m-1,n-k_l+2,k_1-k_l+2, k_2-k_l+2,\ldots,k_{l-1}-k_l+2)])+\]
\[\cdots+e(\KK[\P(m-1,n-1,k_1-1, k_2-l,\ldots,k_{l-1}-1)])\]\[+e(\KK[\P(m-1,n,k_1, k_2,\ldots,k_{l-1})]).\]
In other words, we have
\[e(\KK[\P(m,n,k_1, k_2,\ldots,k_l)])=\]
\[\sum_{j_1=0}^{k_l-1}e(\KK[\P(m-1,n-j_1,k_1-j_1, k_2-j_1,\ldots,k_{l-1}-j_1)]).\]
By iterating formula, we obtain
\[e(\KK[\P(m,n,k_1, k_2,\ldots,k_l)])=\]
\[\sum_{j_1=0}^{k_l-1}(\sum_{j_2=0}^{k_{l-1}-j_1-1}e(\KK[\P(m-2,n-j_1-j_2,k_1-j_1-j_2, k_2-j_1-j_2,\ldots,k_{l-2}-j_1-j_2)]))=\]
\[\cdots\]\[\sum_{j_1=0}^{k_l-1}(\sum_{j_2=0}^{k_l-j_1-1}\cdots (\sum_{j_{l-1}=0}^{k_2-j_1-\cdots -j_{l-2}-1}(e(\KK[\P(m-l+1,n-j_1-\cdots -j_{l-1},k_1-j_1-\cdots -j_{l-1})]))))\]
\[=\sum_{j_1=0}^{k_l-1}(\sum_{j_2=0}^{k_l-j_1-1}\cdots (\sum_{j_{l-1}=0}^{k_2-j_1-\cdots -j_{l-2}-1}({{(m-l+1)+(n-j_1-\cdots-j_{l-1})-2}\choose {(m-l+1)-1}}-\]\[-{{(m-l+1)+(n-j_1-\cdots-j_{l-1})-(k_1-j_1-\cdots-j_{l-1})-3}\choose{(m-l+1)-1}}))).\]
Thus, \[e(\KK[\P(m,n,k_1, k_2,\ldots,k_l)])=\]\[=\sum_{j_1=0}^{k_l-1}(\sum_{j_2=0}^{k_l-j_1-1}\cdots (\sum_{j_{l-1}=0}^{k_2-j_1-\cdots -j_{l-2}-1}({{m+n-l-j_1-\cdots-j_{l-1}-1}\choose {m-l}}-\]\[-{{m+n-l-k_1-2}\choose{m-l}}))).\]
\begin{figure}
    \centering
      \begin{tikzpicture}[domain=0:10]
        \draw[-,black] (0,0) -- (0,9) -- (2,9) -- (2,8) -- (3,8) -- (3,6) -- (4,6) -- (4,5) -- (5,5) -- (5,2) -- (6,2) -- (6,0) -- (0,0);
        \filldraw (6, 2) circle (0.80 pt) node[right] {$(m,k_l)$};
        \filldraw (5, 5) circle (0.80 pt) node[right] {$(m-1, k_{l-1})$};
        \filldraw (4, 6) circle (0.80 pt) node[right] {$(m-2,k_{l-2})$};
        \filldraw (3, 8) circle (0.80 pt) node[right] {$(m-l+1,k_1)$};
        \draw[-,black] (0,-0.5) -- node[near start,below] {$m$} (6,-0.5);
        \draw[-,black] (-0.5,0) -- node[near start,left] {$n$} (-0.5,9);
    \end{tikzpicture}
    \caption{}\label{ex5}
\end{figure}
\end{Example}

One may, of course, approach the computation of the multiplicity in a recursive way for arbitrary convex polyominoes. Finding the appropriate order of the variables in concordance to that one described in \cite{OTH} is not difficult as we will see in the examples below. What is difficult in the general case is to identify the link of a suitable chosen vertex as a simplicial complex of another polyomino related to the original one. We illustrate part of these difficulties in the following examples.
\begin{Example}\label{exx}
Let $\P$ be the convex polyomino of Figure~\ref{ex6}. According to the proof of \cite{OTH}, the generators of $I_{\P}$ form the reduced Gr\"obner basis of $I_{\P}$ with respect to the reverse lexicographical order induced by the following order of variables: $x_{34}>x_{33}>x_{32}>x_{31}>x_{24}>x_{23}>x_{22}>x_{14}>x_{13}>x_{12}>x_{43}>x_{42}>x_{41}>x_{53}>x_{52}$. We consider the vertex $v=(5,3)$. The link of $v$ in $\Delta_{\P}$ is the cone of the vertex $(5,2)$ with the simplicial complex which we may associate to the collection of cells $Q$ displayed in Figure~\ref{ex6} in a similar way which we used for the stack polyominoes. The problem is that the collection $Q$ is no longer a convex polyomino.
\begin{figure}
  \centering
  \begin{tikzpicture}[domain=0:11]
        \filldraw[fill=black!5!white, draw=black] (1,2) rectangle (2,3);
        \filldraw[fill=black!5!white, draw=black] (1,3) rectangle (2,4);
        \filldraw[fill=black!5!white, draw=black] (2,2) rectangle (3,3);
        \filldraw[fill=black!5!white, draw=black] (2,3) rectangle (3,4);
        \filldraw[fill=black!5!white, draw=black] (3,1) rectangle (4,2);
        \filldraw[fill=black!5!white, draw=black] (3,2) rectangle (4,3);
        \filldraw[fill=black!5!white, draw=black] (4,2) rectangle (5,3) circle (0 pt) node[above] {$v$};
        \filldraw[fill=black!5!white, draw=black] (6,3) rectangle (7,4);
        \filldraw[fill=black!5!white, draw=black] (7,3) rectangle (8,4);
        \filldraw[fill=black!5!white, draw=black] (8,2) rectangle (9,3);
        \filldraw (1, 5) circle (0 pt) node[above] {$\P$};
        \filldraw (6, 5) circle (0 pt) node[above] {$Q$};
  \end{tikzpicture}
  \caption{}\label{ex6}
\end{figure}
\end{Example}
\begin{Example}
We consider $\P$ to be the convex polyomino of Figure~\ref{ex7}. In this case the order of the variables is
$x_{24}>x_{23}>x_{22}>x_{14}>x_{13}>x_{12}>x_{43}>x_{42}>x_{41}>x_{33}>x_{32}>x_{31}>x_{53}>x_{52}$. Let $v=(5,3)$. As in Example~\ref{exx}, the link of $v$ is the cone of the vertex $(5,2)$ with the simplicial complex which we may associate to the collection of cells $Q$ displayed in Figure~\ref{ex7}. We notice that the collection $Q$ is not a convex polyomino.
\begin{figure}
  \centering
  \begin{tikzpicture}[domain=0:11]
        \filldraw[fill=black!5!white, draw=black] (1,2) rectangle (2,3);
        \filldraw[fill=black!5!white, draw=black] (1,3) rectangle (2,4);
        \filldraw[fill=black!5!white, draw=black] (2,2) rectangle (3,3);
        \filldraw[fill=black!5!white, draw=black] (3,1) rectangle (4,2);
        \filldraw[fill=black!5!white, draw=black] (3,2) rectangle (4,3);
        \filldraw[fill=black!5!white, draw=black] (4,2) rectangle (5,3) circle (0 pt) node[above] {$v$};
        \filldraw[fill=black!5!white, draw=black] (6,3) rectangle (7,4);
        \filldraw[fill=black!5!white, draw=black] (8,2) rectangle (9,3);
        \filldraw (1, 5) circle (0 pt) node[above] {$\P$};
        \filldraw (6, 5) circle (0 pt) node[above] {$Q$};
  \end{tikzpicture}
  \caption{}\label{ex7}
\end{figure}
\end{Example}

\end{document}